\RequirePackage{fix-cm}
\documentclass[smallextended]{svjour3}       
\smartqed  
\usepackage{graphicx,multicol,amsmath,amssymb,verbatim}
\usepackage{pstricks-add}
\journalname{Mathematische Zeitschrift}
\journalname{\ }
\smartqed  
\usepackage{etoolbox}

\makeatletter
\BeforeBeginEnvironment{proof}{%
  \def\@Opargbegintheorem#1#2#3#4{#4\trivlist
      \item[]{#3#2\@thmcounterend\ }}%
}
\AfterEndEnvironment{proof}{%
  \def\@Opargbegintheorem#1#2#3#4{#4\trivlist
      \item[\hskip\labelsep{#3#1}]{#3(#2)\@thmcounterend\ }}%
}

\newcommand{\Z}{\mathbb{Z}}

\newcommand{\pres}[2]{\langle {#1}\;|\;{#2} \rangle}
\newcommand{\twolinepres}[3]{
\Big \langle {#1} \Big | \begin{array}{ll@{}}
         {#2}\\
         {#3}\end{array} \Big \rangle
}
\newcommand{\threelinepres}[4]{
\Bigg \langle {#1} \Bigg | \begin{array}{ll@{}}
         {#2}\\
         {#3}\\
         {#4}\end{array} \Bigg \rangle
}
\newcommand{\gpres}[1]{\langle {#1} \rangle}
\def\'{^\prime}
\def\"{^{\prime\prime}}
\newcommand{\ra}{\rightarrow}
\def\ra{\rightarrow}
\renewcommand{\gcd}{}

\newtheorem{maintheorem}{Theorem}

\newtheorem{maincorollary}[maintheorem]{Corollary}

\makeatletter
\let\c@proposition\c@theorem
\let\c@corollary\c@theorem
\let\c@lemma\c@theorem
\makeatother

\numberwithin{theorem}{section}
\numberwithin{lemma}{section}
\numberwithin{proposition}{section}
\numberwithin{corollary}{section}

\begin{document}

\title{Efficient Finite Groups Arising in the Study of Relative Asphericity\thanks{Part of this research was carried out during a visit by the second named author to the Department of Mathematics at Oregon State University in July 2014. That visit was financed by a London Mathematical Society Scheme 4 grant (ref.\,41332) and by the OSU College of Science and Department of Mathematics. The second named author would like to thank the LMS and OSU for its support and the OSU Department of Mathematics for its hospitality during that visit.}}


\author{William A. Bogley \and Gerald Williams
}


\institute{William A.Bogley \at
              Department of Mathematics,
              Kidder Hall 368,
              Oregon State University,
              Corvallis,
              OR 97331-4605,
              USA. \\
              \email{ Bill.Bogley@oregonstate.edu }           
           \and
           Gerald Williams\at
              Department of Mathematical Sciences,
              University of Essex,
              Wivenhoe Park,
              Colchester,
              Essex CO4 3SQ,
              U.K.\\
              \email{ gwill@essex.ac.uk }           
}

\date{Received: date / Accepted: date}
\date{\ }

\def\makeheadbox{{%
}}

\maketitle

\begin{abstract}
We study a class of two-generator two-relator groups, denoted \linebreak $J_n(m,k)$, that arise in the study of relative asphericity as groups satisfying a transitional curvature condition. Particular instances of these groups occur in the literature as finite groups of intriguing orders. Here we find infinite families of non-elementary virtually free groups and of finite metabelian non-nilpotent groups, for which we determine the orders. All Mersenne primes arise as factors of the orders of the non-metacyclic groups in the class, as do all primes from other conjecturally infinite families of primes. We classify the finite groups up to isomorphism and show that our class overlaps and extends a class of groups $F^{a,b,c}$ with trivalent Cayley graphs that was introduced by C.M.Campbell, H.S.M.Coxeter, and E.F.Robertson. The theory of cyclically presented groups informs our methods and we extend part of this theory (namely, on connections with polynomial resultants) to ``bicyclically presented groups'' that arise naturally in our analysis. As a corollary to our main results we obtain new infinite families of finite metacyclic generalized Fibonacci groups.
\end{abstract}

\section{Introduction}\label{sec:intro}

An abstract presentation for a finite group must have at least as many defining relations as it has generators. In the limiting case where a finite group admits a presentation that is \textit{balanced}, in the sense that the number of defining relations is equal to the number of generators, the relations interact in economical and interesting ways to ensure finiteness. These interactions are manifested in terms of \textit{identities among relations}~\cite{Sieradski80,Sieradski93} that in turn give rise to \textit{spherical van Kampen diagrams}~\cite{CCH81} or, dually, \textit{spherical pictures}~\cite{BogleyPride93}. For group presentations, whether balanced or not, the term \textit{asphericity} has been used in various ways but always refers to the absence of nontrivial identities among relations. See~\cite{CCH81} for a precise discussion. The purpose of this paper is to demonstrate that interesting groups and presentations occur as aspherical presentations transition to non-aspherical ones by investigating one such family in detail.

We examine the transition to asphericity in the context of \textit{relative presentations} $\pres{H,{\bf y}}{\bf r}$. Here $H$ is a group, $\bf y$ is a set disjoint from $H$, and $\bf r$ is a set of words representing elements in the free product $H \ast F$ where $F$ is the free group with basis $\bf y$. The elements of $H$ that occur in the defining relations $\bf r$ are the \textit{coefficients} and the relative presentation defines the group $G = (H \ast F)/\gpres{\mathbf{r}}$, which is the quotient of the free product by the smallest normal subgroup containing $\mathbf{r}$. Following~\cite{BogleyPride}, the \textit{cellular model} for $\pres{H,{\bf y}}{\bf r}$ is a relative two-complex $(Y,X)$ where $X$ is a $K(H,1)$-complex and the inclusion $X \rightarrow Y$ induces the natural homomorphism $H \rightarrow G$ on fundamental groups. Attaching additional higher dimensional cells to account for proper power relators $r = \dot{r}^{e(r)} \in \mathbf{r}$, if any, we obtain an augmented cellular model $M$. Motivated by~\cite[Theorem 4.1]{BogleyPride}, we say that the relative presentation $\pres{H,{\bf y}}{\bf r}$ is \textit{aspherical} if the second relative homotopy group $\pi_2(M,X)$ is trivial. This formulation of asphericity, which is slightly more general than that given in~\cite{BogleyPride}, is equivalent to saying that the homomorphism $H \ra G$ is injective and that $M$ is a $K(G,1)$-complex.

The papers~\cite{Ahmad,AhmadAlMullaEdjvet14,AldwaikEdjvet14,BaikBogleyPride,BogleyPride,Davidson09,Edjvet94,EdjvetJuhasz14,HowieMetaftsis,Metaftsis} all contain detailed classifications of asphericity as it occurs in classes of one-relator relative presentations of the form $\pres{H,y}{r}$ where the single relator $r$ has free product length four in $H \ast \gpres{y}$. Letting $m$ denote the exponent sum of the letter $y$ in the relator $r$, such a relator can be written in the form $r = y^{m-k}gy^kh$ where $k$ is a non-zero integer and the coefficients $g,h$ lie in $H$. The general approach in all of these papers is to fix $m$ and $k$ and then classify asphericity according to the occurrences of the coefficients $g, h \in H$. The analyses become increasingly complex as the sum $|m-k| + |k|$ increases, and there are exceptional cases where the asphericity status remains unresolved. Nevertheless, some consistent themes have emerged. In Theorems~A and~B of~\cite{HowieMetaftsis,Metaftsis} it is shown that the relative presentations considered there are aspherical unless the coefficients generate a group that is either finite cyclic or else a quotient of a finite triangle group. The same statement holds true for the cases treated in~\cite{BaikBogleyPride,BogleyPride,Davidson09,Edjvet94}, so in all these cases a transition from aspherical presentations to non-aspherical ones occurs when the coefficients generate a group that is either finite cyclic or else a quotient of a finite triangle group. Further, in the cases considered in~\cite{BogleyPride},\cite{BaikBogleyPride},\cite{Edjvet94},\cite{HowieMetaftsis}, if the coefficients satisfy the positive curvature condition $1/\mathrm{o}(g)+1/\mathrm{o}(h)+1/\mathrm{o}(gh^{-1})>1$ (where $h\neq g^{-1}$ and $\mathrm{o}(g)$ denotes the order of $g$ in $H$), then there are constructions of essential spherical pictures that employ the combinatorial structure of the Platonic solids. See~\cite[Theorem 4(3)]{BaikBogleyPride},  \cite[Theorem 3.4 and Figure~6]{BogleyPride}, \cite[Theorem 1.1(a)]{Edjvet94}, \cite[Lemma 10, Lemma 12]{HowieMetaftsis}. Even simpler constructions apply when $g = h$~\cite[page 288]{BaikBogleyPride}, \cite[page 23]{BogleyPride}.

The case where the coefficients generate a finite cyclic group is of general interest for several reasons. First, when the coefficients of a one-relator relative presentation generate a finite cyclic group, we are effectively considering groups that admit (ordinary) two-generator two-relator presentations of the form $\pres{t,y}{t^n,W}$. If a finite group $G$ admits an $m$-generator $p$-relator presentation, then the difference $p-m$ is bounded below by the minimum number of elements required to generate the Schur multiplicator $H_2(G,\Z)$~\cite{Epstein61}. The finite group $G$ is \textit{efficient} if it admits a presentation that realizes this bound. Thus any finite group that admits a balanced presentation has trivial Schur multiplicator and is efficient. The first examples of inefficient groups were exhibited in 1965~\cite{Swan65}, but there are longstanding and ongoing efforts to identify classes of efficient finite groups. See e.~g.~\cite{CCR77} and~\cite[Chapter III]{Johnson76}. Many finite simple groups are efficient~\cite{CHRR14} and it has been suggested that the covering group of every finite simple group should admit a two-generator two-relator presentation~\cite{GKKL07,Wilson06}. In another direction, there is no known bound on the derived length for soluble finite groups that admit two-generator two-relator presentations~\cite{HHKR99}.


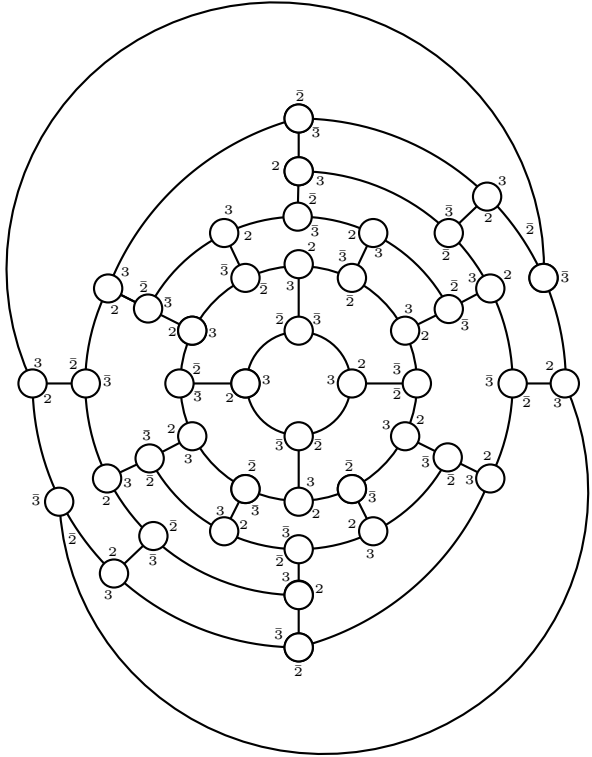
\begin{figure}
\tiny
\begin{center}
\psset{xunit=0.7cm,yunit=0.7cm,algebraic=true,dotstyle=o,dotsize=3pt 0,linewidth=0.8pt,arrowsize=3pt 2,arrowinset=0.25}
\begin{pspicture*}(4,-11)(18,5)
\psline(11,-2)(11,-0.74)
\psline(12,-3)(13.22,-3)
\psline(11,-4)(11,-5.24)
\psline(10,-3)(8.76,-3)
\psline(10,-1)(9.6,-0.15)
\psline(12,-1)(12.4,-0.15)
\psline(13,-4)(13.8,-4.4)
\psline(12,-5)(12.4,-5.8)
\psline(10,-5)(9.6,-5.8)
\psline(9,-4)(8.2,-4.42)
\psline(14.6,-4.8)(13.8,-4.4)
\psline(7.4,-4.8)(8.2,-4.42)
\psline(11,1.02)(11,2.02)
\psline(11,-7)(11,-8)
\psline(7,-3)(6,-3)
\psline(15.02,-3)(16,-3)
\psline(13.82,-0.15)(14.54,0.54)
\psline(8.27,-5.89)(7.53,-6.6)
\parametricplot{-0.03275161250070191}{3.5851350444700296}{1*5.05*cos(t)+0*5.05*sin(t)+10.56|0*5.05*cos(t)+1*5.05*sin(t)+-0.83}
\parametricplot{-3.1042846358836798}{0.4258119532470605}{1*4.97*cos(t)+0*4.97*sin(t)+11.47|0*4.97*cos(t)+1*4.97*sin(t)+-5.05}
\parametricplot{-0.009372188253242264}{1.5707963267948966}{1*5.02*cos(t)+0*5.02*sin(t)+11|0*5.02*cos(t)+1*5.02*sin(t)+-3}
\parametricplot{3.141592653589793}{4.71238898038469}{1*5*cos(t)+0*5*sin(t)+11|0*5*cos(t)+1*5*sin(t)+-3}
\parametricplot{2.675712905439022}{4.71238898038469}{1*4.01*cos(t)+0*4.01*sin(t)+11|0*4.01*cos(t)+1*4.01*sin(t)+-3}
\parametricplot{-0.46364760900080615}{1.5707963267948966}{1*4.02*cos(t)+0*4.02*sin(t)+11|0*4.02*cos(t)+1*4.02*sin(t)+-3}
\parametricplot{1.8276472666413774}{2.779350745109478}{1*5.26*cos(t)+0*5.26*sin(t)+12.34|0*5.26*cos(t)+1*5.26*sin(t)+-3.06}
\parametricplot{4.960678889687757}{5.917383752445073}{1*5.23*cos(t)+0*5.23*sin(t)+9.71|0*5.23*cos(t)+1*5.23*sin(t)+-2.93}
\pscircle(11,-3){0.7}
\rput[tl](11.92,-0.1){$2$}
\rput[tl](11.16,-0.32){$2$}
\rput[tl](9.96,-0.12){$2$}
\rput[tl](14.48,0.22){$2$}
\rput[tl](10.48,1.20){$2$}
\rput[tl](15.63,-2.58){$2$}
\rput[tl](11.92,-5.58){$2$}
\rput[tl](14.46,-4.28){$2$}
\rput[tl](13.20,-3.6){$2$}
\rput[tl](12.1,-2.54){$2$}
\rput[tl](9.88,-5.6){$2$}
\rput[tl](11.24,-5.38){$2$}
\rput[tl](7.42,-6.10){$2$}
\rput[tl](7.3,-5.14){$2$}
\rput[tl](6.20,-3.2){$2$}
\rput[tl](9.62,-3.18){$2$}
\rput[tl](8.54,-1.94){$2$}
\rput[tl](8.52,-3.78){$2$}
\rput[tl](14.84,-0.86){$2$}
\rput[tl](13.29,-2.02){$2$}
\rput[tl](7.46,-1.52){$2$}
\rput[tl](7.66,-0.8){$3$}
\rput[tl](9.6,0.38){$3$}
\rput[tl](14.74,0.92){$3$}
\rput[tl](15.78,-3.28){$3$}
\rput[tl](14.13,-4.74){$3$}
\rput[tl](12.26,-6.14){$3$}
\rput[tl](10.68,-6.58){$3$}
\rput[tl](7.34,-6.92){$3$}
\rput[tl](7.7,-4.8){$3$}
\rput[tl](6.02,-2.52){$3$}
\rput[tl](9.30,-1.96){$3$}
\rput[tl](10.76,-1.06){$3$}
\rput[tl](12.56,-3.70){$3$}
\rput[tl](11.14,-4.78){$3$}
\rput[tl](14.18,-0.92){$3$}
\rput[tl](12.42,-0.44){$3$}
\rput[tl](11.32,0.92){$3$}
\rput[tl](8.84,-4.34){$3$}
\rput[tl](10.26,-1.1){$\bar{2}$}
\rput[tl](9.00,-2.62){$\bar{2}$}
\rput[tl](10.56,-1.68){$\bar{2}$}
\rput[tl](11.18,0.6){$\bar{2}$}
\rput[tl](13.68,-0.46){$\bar{2}$}
\rput[tl](13.84,-1.04){$\bar{2}$}
\rput[tl](15.26,0.04){$\bar{2}$}
\rput[tl](10.94,2.56){$\bar{2}$}
\rput[tl](6.68,-2.52){$\bar{2}$}
\rput[tl](8,-1.06){$\bar{2}$}
\rput[tl](8.56,-5.6){$\bar{2}$}
\rput[tl](11.86,-4.46){$\bar{2}$}
\rput[tl](12.76,-3.1){$\bar{2}$}
\rput[tl](15.20,-3.24){$\bar{2}$}
\rput[tl](8.12,-4.74){$\bar{2}$}
\rput[tl](10.56,-6.26){$\bar{2}$}
\rput[tl](10.91,-8.34){$\bar{2}$}
\rput[tl](13.78,-4.72){$\bar{2}$}
\rput[tl](6.66,-5.84){$\bar{2}$}
\rput[tl](11.24,1.86){$\bar{3}$}
\rput[tl](13.74,0.38){$\bar{3}$}
\rput[tl](14.50,-2.84){$\bar{3}$}
\rput[tl](12.76,-2.62){$\bar{3}$}
\rput[tl](10.55,-4.06){$\bar{3}$}
\rput[tl](10.10,-5.30){$\bar{3}$}
\rput[tl](5.98,-5.1){$\bar{3}$}
\rput[tl](15.92,-0.88){$\bar{3}$}
\rput[tl](11.74,-0.48){$\bar{3}$}
\rput[tl](11.22,0.00){$\bar{3}$}
\rput[tl](12.30,-5.04){$\bar{3}$}
\rput[tl](8.06,-3.88){$\bar{3}$}
\rput[tl](7.32,-2.84){$\bar{3}$}
\rput[tl](11.26,-1.68){$\bar{3}$}
\rput[tl](10.54,-7.64){$\bar{3}$}
\rput[tl](8.18,-6.24){$\bar{3}$}
\rput[tl](9.44,-5.32){$3$}
\rput[tl](10.04,-4.46){$\bar{2}$}
\rput[tl](11.30,-6.8){$2$}
\rput[tl](13.3,-4.38){$\bar{3}$}
\rput[tl](14.04,-1.78){$\bar{3}$}
\rput[tl](10.31,-2.84){$3$}
\rput[tl](11.52,-2.84){$3$}
\rput[tl](11.88,-1.34){$\bar{2}$}
\rput[tl](8.46,-1.36){$\bar{3}$}
\rput[tl](9.52,-0.74){$\bar{3}$}
\rput[tl](12.98,-1.5){$3$}
\rput[tl](11.28,-4.06){$\bar{2}$}
\pscircle(11,-3){1.57}
\parametricplot{3.6109382066235782}{5.81953769817878}{1*3.14*cos(t)+0*3.14*sin(t)+11|0*3.14*cos(t)+1*3.14*sin(t)+-3}
\psline(11,-7.01)(11,-6.14)
\parametricplot{0.4636476090008061}{2.677945044588987}{1*3.16*cos(t)+0*3.16*sin(t)+11|0*3.16*cos(t)+1*3.16*sin(t)+-3}
\psline(10.98,0.16)(11,1.02)
\psline(9,-2)(8.17,-1.58)
\psline(8.17,-1.58)(7.42,-1.2)
\psline(13.82,-1.59)(14.59,-1.18)
\psline(13,-2)(13.82,-1.59)
\psline(8.17,-1.58)(7.42,-1.2)
\rput[tl](10.68,-5.7){$\bar{3}$}
\rput[tl](9.02,-3.14){$\bar{3}$}
\pscircle[fillcolor=white,fillstyle=solid](10,-3){0.2}
\pscircle[fillcolor=white,fillstyle=solid](10,-3){0.2}
\pscircle[fillcolor=white,fillstyle=solid](11,-2){0.2}
\pscircle[fillcolor=white,fillstyle=solid](12,-3){0.2}
\pscircle[fillcolor=white,fillstyle=solid](11,-4){0.2}
\pscircle[fillcolor=white,fillstyle=solid](11,-0.74){0.2}
\pscircle[fillcolor=white,fillstyle=solid](13.22,-3){0.2}
\pscircle[fillcolor=white,fillstyle=solid](11,-5.24){0.2}
\pscircle[fillcolor=white,fillstyle=solid](8.76,-3){0.2}
\pscircle[fillcolor=white,fillstyle=solid](9,-2){0.2}
\pscircle[fillcolor=white,fillstyle=solid](10,-1){0.2}
\pscircle[fillcolor=white,fillstyle=solid](12,-1){0.2}
\pscircle[fillcolor=white,fillstyle=solid](13,-2){0.2}
\pscircle[fillcolor=white,fillstyle=solid](13,-4){0.2}
\pscircle[fillcolor=white,fillstyle=solid](12,-5){0.2}
\pscircle[fillcolor=white,fillstyle=solid](10,-5){0.2}
\pscircle[fillcolor=white,fillstyle=solid](9,-4){0.2}
\pscircle[fillcolor=white,fillstyle=solid](9.6,-0.15){0.2}
\pscircle[fillcolor=white,fillstyle=solid](12.4,-0.15){0.2}
\pscircle[fillcolor=white,fillstyle=solid](13.8,-4.4){0.2}
\pscircle[fillcolor=white,fillstyle=solid](12.4,-5.8){0.2}
\pscircle[fillcolor=white,fillstyle=solid](9.6,-5.8){0.2}
\pscircle[fillcolor=white,fillstyle=solid](8.2,-4.42){0.2}
\pscircle[fillcolor=white,fillstyle=solid](14.6,-1.2){0.2}
\pscircle[fillcolor=white,fillstyle=solid](15.02,-3){0.2}
\pscircle[fillcolor=white,fillstyle=solid](14.6,-4.8){0.2}
\pscircle[fillcolor=white,fillstyle=solid](7.42,-1.2){0.2}
\pscircle[fillcolor=white,fillstyle=solid](7,-3){0.2}
\pscircle[fillcolor=white,fillstyle=solid](7.4,-4.8){0.2}
\pscircle[fillcolor=white,fillstyle=solid](11,1.02){0.2}
\pscircle[fillcolor=white,fillstyle=solid](11,2.02){0.2}
\pscircle[fillcolor=white,fillstyle=solid](11,-7){0.2}
\pscircle[fillcolor=white,fillstyle=solid](11,-8){0.2}
\pscircle[fillcolor=white,fillstyle=solid](6,-3){0.2}
\pscircle[fillcolor=white,fillstyle=solid](16,-3){0.2}
\pscircle[fillcolor=white,fillstyle=solid](13.82,-0.15){0.2}
\pscircle[fillcolor=white,fillstyle=solid](14.54,0.54){0.2}
\pscircle[fillcolor=white,fillstyle=solid](15.6,-1){0.2}
\pscircle[fillcolor=white,fillstyle=solid](8.27,-5.89){0.2}
\pscircle[fillcolor=white,fillstyle=solid](7.53,-6.6){0.2}
\pscircle[fillcolor=white,fillstyle=solid](6.5,-5.24){0.2}
\pscircle[fillcolor=white,fillstyle=solid](15.6,-1){0.2}
\pscircle[fillcolor=white,fillstyle=solid](11,2.02){0.2}
\pscircle[fillcolor=white,fillstyle=solid](11,-8){0.2}
\pscircle[fillcolor=white,fillstyle=solid](11,1.02){0.2}
\pscircle[fillcolor=white,fillstyle=solid](11,-7.01){0.2}
\pscircle[fillcolor=white,fillstyle=solid](11,-6.14){0.2}
\pscircle[fillcolor=white,fillstyle=solid](10,-5){0.2}
\pscircle[fillcolor=white,fillstyle=solid](12,-5){0.2}
\pscircle[fillcolor=white,fillstyle=solid](9,-2){0.2}
\pscircle[fillcolor=white,fillstyle=solid](13.82,-1.59){0.2}
\pscircle[fillcolor=white,fillstyle=solid](8.17,-1.58){0.2}
\pscircle[fillcolor=white,fillstyle=solid](10.98,0.16){0.2}
\end{pspicture*}
\normalsize
\end{center}
\vspace{-1cm}
\caption{Spherical picture over $J_6(3,1)$.\label{fig:sphere}}
\end{figure}

A second source of general interest in presentations of the form $\pres{t,y}{t^n,W}$ comes from the study of cyclically presented groups. Let $w$ be a word representing an element in the free group $F_n$ with basis $\{x_0, \ldots , x_{n-1}\}$ and let $\theta$ be the \textit{shift automorphism} of $F_n$ given by $\theta(x_i) = x_{i+1}$ with subscripts modulo $n$. The \textit{cyclically presented group} $G = G_n(w)$ is defined by the presentation
\begin{equation}
\pres{x_0, x_1, \ldots, x_{n-1}}{w, \theta(w), \ldots, \theta^{n-1}(w)}\label{eq:cycpresdef}
\end{equation}
so the shift defines an automorphism $\theta \in \mathrm{Aut}(G_n(w))$ of exponent $n$. It is well-known that by setting $x_i=t^iyt^{-i}$ and $t^n=1$, the split extension $E_n(w) = G_n(w) \rtimes_\theta \Z_n$ by the cyclic group of order $n$ admits a two-generator two-relator presentation of the form $\pres{t,y}{t^n,W}$. For the cyclically presented groups considered in each of~\cite{CampbellRobertsonLMS75,CampbellRobertsonEMS75,CampbellThomas86,EdjvetJuhasz14,JWW,Prishchepov95,Thomas91,WilliamsRevisited}, it is possible to express the relator $W$ for the split extension as a length four word in the free product $\gpres{t}*\gpres{y}$.

To show that a relative presentation is not aspherical, a direct approach is to construct a suitable spherical picture representing an essential spherical map into the cellular model. This requires that certain relations hold amongst the coefficients. The precise statement is given in~\cite[(1.5)]{BogleyPride}, but the key point is that for the purpose of constructing pictures, the more relations, the better. In particular, the case where the coefficients $g,h$ generate a finite cyclic group presents many opportunities for these relations to arise. Another way to show that a relative presentation is not aspherical is to show that the group that it defines contains an element of finite order that is not conjugate to an element of the coefficient group~\cite[(0.4)]{BogleyPride}. For a relative presentation of the form $\pres{H,y}{y^{m-k}gy^kh}$, if $g = h^{-1}$ has finite order in $H$, then a suitable power of $y^{m-k}$ is equal to $y^{k}$. Barring exceptional circumstances, this means that $y$ is an element of finite order that is not conjugate to an element of $H$. Further, when $(m,k)=1$ it follows from a result of Pride~\cite{PrideExactlyTwo} that this relative presentation defines a finite group. We record this known result explicitly in Corollary~\ref{cor:h=ginversefinite}.

For relators of the form $r = y^{m-k}gy^kh$, these considerations point to a limiting case where the asphericity status of the relative presentation $\pres{H,{\bf y}}{\bf r}$ is not obvious. Namely, when $H=\gpres{g,h}$ is finite cyclic of order~$n$, say, $g\neq h^{\pm 1}$, and $1/\mathrm{o}(g) + 1/\mathrm{o}(h) + 1/\mathrm{o}(gh^{-1}) =1$. If $gh^{-1}$ does not generate $H$ then both finite and infinite groups can arise and there are cases where the presentation is not aspherical and cases where the asphericity status remains unresolved (for example, case~E3 of~\cite[page~280]{BaikBogleyPride}), though we do not know of any examples where the presentation is known to be aspherical. In this paper we consider the complementary condition, where $gh^{-1}$ generates $H$. The only cases that meet all these requirements are those that satisfy $n = 4$ and $\{\mathrm{o}(g), \mathrm{o}(h)\} = \{2,4\}$, or else $n = 6$ and $\{\mathrm{o}(g), \mathrm{o}(h)\} = \{2, 3\}$. This leads to the ordinary group presentations
\[J_n(m,k) = \pres{t,y}{t^n, y^{m-k}t^3y^kt^2}\]
where $n = 4$ or $6$. There are variations of the $J_n(m,k)$ that also arise along the transition to asphericity, including the exceptional cases occurring in~\cite{BaikBogleyPride,Davidson09,Edjvet94,HowieMetaftsis,Metaftsis}, for which the asphericity status is unresolved. These and other families of presentations arising within general studies of relative asphericity are deserving of investigation for their group-theoretic, combinatorial, and number-theoretic properties. The groups and presentations $J_n(m,k)$ are the focus of our work in this paper. Our first result is that very few aspherical relative presentations occur in the situations that we are considering.
\begin{maintheorem}\label{mainthm:asphericity} Let $n = 4$ or $6$ and let $g,h$ be elements of a group $H$ such that ${1}/{\mathrm{o}(g)} +$ ${1}/{\mathrm{o}(h)} +{1}/{\mathrm{o}(gh^{-1})} = 1$ and $\gpres{g,h} = \gpres{gh^{-1}} \cong \Z_n$, $h\neq g^{-1}$. Given integers $m$ and $k$, the following are equivalent:

\begin{enumerate}
\item[1.] the relative presentation $\pres{H,y}{y^{m-k}gy^kh}$ is aspherical;
\item[2.] $m = \pm1$ and $k$ is congruent to either $0$ or $m$ modulo $n$;
\item[3.] for the group $G$ defined by the relative presentation $\pres{H,y}{y^{m-k}gy^kh}$, the natural homomorphism $H \ra G$ is an isomorphism.
\end{enumerate}
\end{maintheorem}
Theorem~\ref{mainthm:asphericity} implies that except for the special cases detailed in condition A.2, the (ordinary) presentations $J_4(m,k)$ and $J_6(m,k)$ support non-trivial identities amongst the relations, and thus nontrivial spherical pictures. With very few exceptions, e.g. for $J_6(2,-1)$ [1, page 71], these pictures are so far unseen. We give a (relative) spherical picture for $J_6(3,1)$ in Figure~\ref{fig:sphere}. Here the corner labels $2,3,\bar{2},\bar{3}$ correspond to $t^2,t^3,t^{-2},t^{-3}$, respectively, and an empty corner label corresponds to $t^0=1$. Each edge in the picture is implicitly labelled by $y$ and oriented transversely so that reading clockwise around any vertex gives (a cyclic permutation of) the relator $y^2t^3yt^2$ or its inverse. See any of~\cite{BaikBogleyPride,BogleyPride,Edjvet94,HowieMetaftsis} for details about relative spherical pictures and other examples.

Partial results in the direction of Theorem~\ref{mainthm:asphericity} were initially obtained using coset enumerations, which produced the following rather puzzling group order calculations. As we shall see, these orders grow exponentially with $m$ and, further, that the presentations are \em short \em in the sense of~\cite{BGKLP97}.

\begin{multicols}{2}
\begin{itemize}
\item[ ] $|J_4(4,1)| = 272$~\cite[Page 291]{BaikBogleyPride}
\item[ ] $|J_4(5,1)| = 820$~\cite[Page 186]{HowieMetaftsis}
\item[ ] $|J_4(5,2)| = 500$~\cite[Page 184]{HowieMetaftsis}
\item[ ] $|J_6(3, 1)| = 342$~\cite[Page 25]{BogleyPride}
\item[ ] $|J_6(4,1)| = 4632$~\cite[Page 292]{BaikBogleyPride}
\item[ ] $|J_6(5,1)| = 38010$~\cite[Page 187]{HowieMetaftsis}
\item[ ] $|J_6(5,2)| = 12090$~\cite[Page 184]{HowieMetaftsis}
\end{itemize}
\end{multicols}

The next result provides that if $m \neq 0$ and $n = 4$ or $6$, then the groups $J_n(m, k)$ are either finite cyclic, finite metacyclic, finite metabelian (and not metacyclic), or else virtually free and non-elementary. These groups therefore satisfy a strong form of the Tits alternative. For a group $G$, $G'$ denotes the derived subgroup and $G'' = (G')'$ denotes the second derived subgroup. The group $G$ is metabelian if $G'' = 1$ and it is metacyclic if it has a cyclic normal subgroup $C$ such that $G/C$ is cyclic.

\begin{maintheorem}\label{mainthm:structure} Given $n = 4$ or $6$ and $m \neq 0$, let $J = J_n(m,k)$.
\begin{enumerate}
\item[(a)] The metabelian quotient $J/J''$ is finite.
\item[(b)] The derived group $J'$ is trivial if and only if one of $k$, $m-k$, or $m-2k$ is congruent to zero modulo $nm$, in which case $J \cong \Z_{nm}$.
\item[(c)] If $m$ and $k$ are relatively prime, then $J$ is metabelian, and hence finite.
In this case if $m-2k\not \equiv 0$~mod~$n$ then $J'$ is cyclic and $J$ is metacyclic, and if $m-2k\equiv 0$~mod~$n$ then $J'$ is the direct product of two cyclic groups of equal order and $J$ is not metacyclic.
\item[(d)] If $m$ and $k$ are not relatively prime and $J'$ is non-trivial, then $J{''}$ is a finitely generated non-abelian free group, so $J$ is virtually free and non-elementary.
\end{enumerate}
\end{maintheorem}

All of the groups $J=J_n(m,k)$ have cyclic abelianizations, so those that are non-cyclic are also non-nilpotent. Theorem~\ref{mainthm:structure}(c) validates our general approach, providing an infinite family of two-generator two-relator finite groups, which therefore have trivial multiplicator and are efficient. As a further application, when $n = 4$ or $6$ and $m \equiv \pm 1 \mod n$, the group $J_n(m,k)$ contains a cyclically presented group as a normal subgroup of index $n$. In this way we obtain new infinite families of finite \textit{generalized Fibonacci groups}, all of which turn out to be metacyclic. See Corollary~\ref{maincor:Fibonacci} in Section~\ref{sec:Fibonacci}.

The theory of cyclically presented groups plays a central role in our arguments.  Letting $a_i$ be the exponent sum of $x_i$ in $w$, the \textit{representer polynomial} of $w$ is $f_w(x) = \sum_{i=0}^{n-1} a_i x^i$. As described in~\cite[pages~77--82]{Johnson76}, a fundamental tool in studying cyclically presented groups is the fact that the relation matrix of the presentation $G_n(w)$ is the circulant matrix $C$ whose first row is $(a_0\; a_1\; \ldots \; a_{n-1})$ and so the abelianization $G_n(w)^\mathrm{ab}$ is finite if and only if $C$ is non-singular, in which case its order is $|\mathrm{det}(C)|$. Further, we have that $\mathrm{det} (C)= \prod_{\lambda^n=1} f_w(\lambda)$ (see, for example, \cite[page~75]{Davis}) so, in particular, $G_n(w)^\mathrm{ab}$ is finite if and only if $f_w(\lambda) \neq 0$ for each complex $n$th root of unity~$\lambda$. If $f,g\in \Z [x]$ are polynomials and $g$ is monic, then the resultant $\mathrm{Res}(f,g)$ satisfies
\begin{equation}\label{eq:resultant}
|\mathrm{Res}(f(x), g(x))| = \prod_{g(\lambda) = 0} |f(\lambda)|
\end{equation}
(see~\cite[Section~7.4]{Cohn} or~\cite[page~75]{Davis}) so when $|G_n(w)^\mathrm{ab}|<\infty$ we have
\begin{equation}\label{eq:representer}
|G_n(w)^\mathrm{ab}| = |\mathrm{Res}(f_w(x), x^n-1)| = \prod_{\lambda^n = 1} |f_w(\lambda)|.
\end{equation}
We generalize these techniques to a class of \textit{bicyclically presented groups} (see Section~\ref{sec:orderabelianization}) to calculate the order of the second derived quotients $J'/J''$ (and hence the order of $J$) in terms of polynomial resultants.  Note that by replacing $y$ by $y^{-1}$, we have $J_n(m,k) \cong J_n(-m,-k)$, so without loss of generality we can assume that $m \geq 0$. We will also show, in Lemmas~\ref{lem:conjugacyrelation4} and~\ref{lem:conjugacyrelation6}, that the parameter $k$ can be taken modulo $nm$. By Theorem~\ref{mainthm:structure} if $n=4$ or $6$ the derived subgroup $J'/J''$ is finite so we define
\[a_n(m,k) = |J_n(m,k)'/J_n(m,k)''|\]
and thus $|J_n(m,k)| = nma_n(m,k)$. It will follow from Corollary~\ref{cor:freeproduct} that if $d = \gcd(m,k) \geq 1$ then $a_n(m,k) = da_n\left({m}/{d},{k}/{d}\right)$.

\begin{maintheorem}\label{mainthm:orders}
Suppose that $m \geq 1$ and $n = 4$ or $6$. If $\gcd(m,k) = 1$, then
\begin{alignat}{1}
a_4(m,k) &= |\mathrm{Res}(1+x^m-x^k,1+x^{2m})|\\
\nonumber &= \prod_{\lambda^{2m} = -1} |1 + \lambda^m - \lambda^k|\\
\nonumber &= 2^m + 1 - 2(\sqrt{2})^m\cos((2k-m)\pi/4)
\intertext{and}
a_6(m,k) &= |\mathrm{Res}(1+x^m+x^{2m}-x^k-x^{m+k},1+x^{3m})|\\
\nonumber &= \prod_{\lambda^{3m} = -1} |1 + \lambda^m +\lambda^{2m} - \lambda^{k}(1 + \lambda^m)|\\
\nonumber &= 3^m+4^m-2(2\sqrt{3})^m\cos((2k-m)\pi/6).
\end{alignat}
\end{maintheorem}

It follows from Theorem~\ref{mainthm:orders} that the squares of all Mersenne primes $2^p-1$ occur as particular values of $a_4(m,k)$. Theorem~\ref{thm:(m-2k)=0mod4} further implies that for each Mersenne prime $2^p-1$, the elementary abelian group of order $(2^p-1)^2$ occurs amongst our groups $J_4(m,k)'$. See the discussion following the proof of Theorem~\ref{thm:requiredres} for further details. Just as the Mersenne prime conjecture asserts that there are infinitely many Mersenne primes, it is expected (see~\cite[page~418]{ESTW07}) that there are infinitely many primes of the form $4^p-3^p$. It follows from Theorem~\ref{mainthm:orders} that the squares of each of these primes (with $p\geq 7$) occur as particular values of $a_6(m,k)$. Theorem~\ref{thm:(m-2k)=0mod6} further implies that for each prime of the form $4^p-3^p$, the elementary abelian group of order $(4^p-3^p)^2$ occurs amongst our groups $J_6(m,k)'$.

The groups $F^{a,b,c}$ defined by the presentations $\pres{R,S}{R^2, RS^aRS^bRS^c}$ first appeared in~\cite{CCR77}. Interest in these groups was motivated in part by studies of $0$-symmetric trivalent Cayley graphs. The finite groups of the form $F^{a,b,c}$ were classified in a series of papers culminating in~\cite{HRS06}. For the groups $F^{a,b,c}$, a connection to Mersenne primes and to related families of large primes was described in~\cite{Perkel85}. The coincidences of the orders of $J_4(m,k)$ (as in Theorem~\ref{mainthm:orders}) and $F^{m,k,m-k}$~\cite[Theorem 7.2]{CCR77} suggested that these groups might be isomorphic. Indeed this is the case and so the presentations $J_4(m,k)$ provide new and  (except for the case $|m| \leq 2$) shorter presentations for the groups $F^{m,k,m-k}$. Moreover, the connection to trivalent Cayley graphs extends to the groups $J_6(m,k)$.
\begin{theorem}\label{thm:abc} For all $m$ and $k$, the group $J_4(m,k)$ is isomorphic to $F^{m,k,m-k}$. The group $J_6(m,k)$ has a presentation $\pres{R,S}{R^2, RS^m(RS^kRS^{m-k})^2}$.
\end{theorem}
The fact that finite generalized Fibonacci groups arise as subgroups of certain groups $J_4(m,k)$ reveals a connection between the groups $F^{a,b,c}$ and generalized Fibonacci groups that has not previously been noted.

There are many isomorphisms amongst the finite groups $J_n(m,k)$, so we carefully enumerate the coincidences. The following result implies that when $(m,k) = 1$, each group $J_4(m,k)$ is isomorphic to $J_4(m,1)$ or $J_4(m,-1)$. Still with $(m,k) = 1$, each group $J_6(m,k)$ is isomorphic to one of $J_6(m,1)$, $J_6(m,-1)$, or $J_6(m,3)$.

\begin{maintheorem}\label{mainthm:Isos} For $n = 4$ or $6$, let $m_1, m_2 \geq 1$, and let $k_1, k_2$ be given so that each $\gcd(m_i,k_i) = 1$. Then $J_n(m_1,k_1) \cong J_n(m_2,k_2)$ if and only if $m_1 = m_2$ and either $k_1 \equiv k_2$  or $k_1+k_2 \equiv m_1$~mod~$n$.
\end{maintheorem}

In proving our results for the groups $J_n(m,k)$, we generally handle the cases $n = 4$ and $n=6$ separately. Nevertheless, in every case the overall strategy is identical for both families. It is therefore reasonable to expect that our arguments could be extended to obtain more general results. However, the complexity of the arguments for the case $n=6$ is generally an order of magnitude greater than for $n=4$, so generalizations will not be straightforward or automatic.

We conclude this introduction with an outline of the paper. In Section~\ref{sec:Asp} we prove Theorem~\ref{mainthm:asphericity} by first showing that $y$ is an element of order $nm$ in $J_n(m,k)$ ($n=4$ or $6$). In Section~\ref{sec:derived} we obtain
bicyclic presentations for the derived subgroup $J_n(m,k)'$ and we show that if $(m,k) = 1$ then the derived subgroup is finite abelian and usually cyclic. In Section~\ref{sec:orderabelianization} we generalize~(\ref{eq:representer}) to give a formula for the order of the abelianization of certain bicyclically presented groups in terms of polynomial resultants. With this we obtain the group orders given in Theorem~\ref{mainthm:orders} and describe some prime factors of the numbers $a_n(m,k)$. In Section~\ref{sec:isoms} we prove the isomorphism theorem, Theorem~\ref{mainthm:Isos}, and prove the connection between the groups $J_n(m,k)$ and the groups $F^{a,b,c}$, Theorem~\ref{thm:abc}. Section~\ref{sec:noncyclicJ'}  examines the case where $(m,k) = 1$ and $n$ is a divisor of $m-2k$; here we show that the derived subgroup $J_n(m,k)'$ is the direct product of two subgroups of equal order, thus completing the proof of Theorem~\ref{mainthm:structure}. In Section~\ref{sec:Fibonacci} we apply our results to obtain new infinite families of finite metacyclic generalized Fibonacci groups (Corollary~\ref{maincor:Fibonacci}).


\section{Torsion and asphericity}\label{sec:Asp}

We are interested in the groups $J_n(m,k) = \pres{t,y}{t^n, y^{m-k}t^3y^kt^2}$ when $n = 4$ or $6$. The group $J = J_n(m,k)$ abelianizes to $J/J' \cong \mathbb{Z}_{nm}$, generated by $yJ'$. The following conjugacy relations are fundamental to all that follows, and lead directly to the proof of Theorem~\ref{mainthm:asphericity}.

\begin{lemma}\label{lem:conjugacyrelation4}
The relation $t^{-1}y^mt = y^{-k}t^{-1}y^k$ holds in the group $J_4(m,k)$. In particular $y^{4m} = 1$ in $J_4(m,k)$, which therefore splits as a semidirect product $$J_4(m,k)\cong J_4(m,k)'\rtimes \pres{y}{y^{4m}}.$$
\end{lemma}

\begin{proof} The relation $y^{m-k}t^3y^kt^2 = 1$ can be rewritten as $t^2y^m = y^{-k}t^{-3}y^k$ and as $y^mt^3 = y^kt^{-2}y^{-k}$. Since $t^4 = 1$, this means that $t^2y^m = y^{-k}ty^k$ and $(y^mt^3)^2 = 1$. Thus $1 = (y^mt^3)(y^mt^3) = y^mty^{-k}ty^kt^{-1}$, whence the result.
\end{proof}

\begin{lemma}\label{lem:conjugacyrelation6}
The relation $y^mty^{-m}=t^{-1}y^mt$ holds in the group $J_6(m,k)$. In particular $y^{6m} = 1$ in $J_6(m,k)$, which therefore splits as a semidirect product $$J_6(m,k)\cong J_6(m,k)'\rtimes \pres{y}{y^{6m}}.$$ The group $J_6(m,k)$ admits an expanded presentation
\[J_6(m,k) = \pres{t,y}{t^6, y^{m-k}t^3y^kt^2, y^mty^{-m}t^{-1}y^{-m}t, (t^2y^m)^2, y^{6m}}.\]
\end{lemma}

\begin{proof} The relation $y^{m-k}t^3y^kt^2 = 1$ can be rewritten as $t^2y^m = y^{-k}t^{-3}y^k$ and as $y^mt^3 = y^kt^{-2}y^{-k}$. Since $t^6 = 1$, this means that $(t^2y^m)^2 = 1$ and $(y^mt^3)^3 = 1$. Thus $t^2y^mt^2 = y^{-m}$ and $1 = (y^mt^3)(y^mt^3)(y^mt^3) = y^mty^{-m}ty^mt^3 = y^mty^{-m}t^{-1}y^{-m}t$, whence the result. \end{proof}

Note in particular that when $m \neq 0$, the parameter $k$ can be interpreted modulo $|nm|$. The next result indicates when the groups $J_n(m,k)$ are abelian. The converse statement will be proved in Section~\ref{sec:orderabelianization}, giving Theorem~\ref{mainthm:structure}(b).

\begin{lemma}\label{lem:abelianJ(m,k)}Let $n = 4$ or $6$.
If $m=0$, or $m\equiv k$~mod~$nm$, or $k\equiv 0$~mod~$nm$, or $m\equiv 2k$~mod~$nm$ then $J_n(m,k)\cong \Z_{nm}$.
\end{lemma}

\begin{proof} This amounts to showing that $J_n(m,k)$ is abelian under the stated conditions. If $m=0$ then Lemmas~\ref{lem:conjugacyrelation4} and~\ref{lem:conjugacyrelation6} show that $t=1$ in $J_n(m,k) \cong \Z$. Lemmas~\ref{lem:conjugacyrelation4} and~\ref{lem:conjugacyrelation6} further imply that we can add the relation $y^{nm} = 1$ to the defining relations for $J_n(m,k)$. Therefore if $k$ or $m-k$ is divisible by $nm$ then the relation $y^{m-k}t^3y^kt^2 = 1$ implies that $J_n(m,k)$ is cyclic because $t^5 = t^{\pm 1}$. If $m \equiv 2k \mod nm$, then we have the relation $y^kt^3y^kt^2 = 1$ and so $1 = (y^kt^2y^kt^3)(t^{-2}y^{-k}t^{-3}y^{-k}) = y^kt^2y^kty^{-k}t^{-3}y^{-k}$. Thus we can further add the relation $y^kty^{-k}t^{-1} = 1$. With this, the relation $y^{m-k}t^3y^kt^2 = 1$ becomes $y^{m}t^5 = 1$ and again $J_n(m,k)$ is cyclic. \end{proof}

\begin{proof}[Proof of Theorem~\ref{mainthm:asphericity}.]
Suppose that $n = 4$ or $6$ and that $g,h \in H$ where $\gpres{g,h} \cong \Z_n$, $h\neq g^{-1}$ and
 ${1}/{\mathrm{o}(g)} + {1}/{\mathrm{o}(h)} +{1}/{\mathrm{o}(gh^{-1})} = 1$. Given integers $m$ and $k$, let $G$ be the group defined by the relative presentation $\pres{H,y}{y^{m-k}gy^kh}$. Either $n = 4$ and $\{\mathrm{o}(g),\mathrm{o}(h)\} = \{2,4\}$ or else $n = 6$ and $\{\mathrm{o}(g),\mathrm{o}(h)\} = \{2,3\}$. In any case, there is a homomorphism $J = J_n(m,k) \ra G$ that carries the cyclic subgroup $\gpres{t}$ of $J = \pres{t,y}{t^n,y^{m-k}t^3y^kt^2}$ onto the subgroup $\gpres{g,h}$ of $H$. (For example, if $n=6$, $\mathrm{o}(g) = 3$, and $\mathrm{o}(h) = 2$, then map $y \in J$ to $y^{-1} \in G$ and $t \in J$ to $gh^{-1} \in H$; the word $y^{m-k}t^3y^kt^2$ then maps to $y^{-(m-k)}h^{-1}y^{-k}g^{-1} = (gy^khy^{m-k})^{-1} = 1$ in $G$.)

Assume now that the relative presentation $\pres{H,y}{y^{m-k}gy^kh}$ is aspherical. This implies that the natural homomorphism $H \ra G$ is injective and so the fact that the elements $g$ and $h$ have distinct orders in $H$ implies that $m \neq 0$. With this, the element $t$ has order exactly $n$ in $J$ and the cyclic subgroup $\gpres{t}$ of $J$ maps isomorphically onto $\gpres{g,h} \leq H$. Thus we have a decomposition of $G$ as a free product amalgamated over a cyclic group of order $n$: $G \cong H \ast_{\gpres{g,h} = \gpres{t}} J = H \ast_{\Z_n} J$. Now asphericity further implies that each finite subgroup of $G$ is conjugate to a subgroup of $H$~\cite[(0.4)]{BogleyPride} and so each finite subgroup of $J$ is therefore conjugate to a subgroup of $\gpres{t} \cong \Z_n$. By Lemmas~\ref{lem:conjugacyrelation4} and~\ref{lem:conjugacyrelation6}, the element $y$ has order precisely $|nm|$ in $J$, so we conclude that $m = \pm 1$. Further, in the presentation for $J = J_n(m,k) = \pres{t,y}{t^n, y^{m-k}t^3y^kt^2}$, we can interpret the parameter $k$ modulo $n$, so we are left to consider the groups $J_4(\pm 1,k)$ for $k = 0,1,2,3$ and $J_6(\pm 1,k)$ for $k = 0,1,2,3,4,5$. These groups are all finite (e.g. using computer-assisted coset enumeration~\cite{GAP}), having orders $n, 20, 42$, or $78$, and the order is $n$ precisely when $k \equiv 0$ or $m$ modulo $n$. Therefore, if $\pres{H,y}{y^{m-k}gy^kh}$ is aspherical, then $m = \pm 1$ and $k$ is congruent to either $0$ or $m$ modulo $n$.

Next, if $m = \pm 1$ and $k$ is congruent to either $0$ or $m$ modulo $n$, then the previous remarks show that the natural map $\Z_n \cong \gpres{t} \ra J$ is an isomorphism, and hence the natural map of $H$ into $G \cong H \ast_{\Z_n} J$ is an isomorphism.

Finally, suppose that $H \ra G$ is an isomorphism. The fact that $g,h \neq 1$ in $G$ implies that $t \neq 1$ in $J$, so that $m \neq 0$ and $G \cong H \ast_{\Z_n} J$ as before. Because $G = H$ we have $J = \Z_n$ and so $m = \pm 1$ and $k \equiv 0$ or $m$ as before. Let $Z$ be a $K(\Z_n,1)$-complex with two-skeleton modeled on the presentation $\pres{t}{t^n}$ for $\Z_n$ and let $W = Z \vee S^1_y \cup c^2$ be obtained by attaching a two-cell with boundary representing the word $y^{m-k}t^3y^kt^2$. Thus by hypothesis $\Z_n \cong \pi_1(Z) \ra \pi_1(W)$ is an isomorphism. Let $X$ be a $K(H,1)$-complex containing $Z$ as a subcomplex and let $Y = X \cup W$ so that $X \cap W = Z$. Since the relator $r = y^{m-k}gy^kh$ is not a proper power, it is necessary and sufficient to prove that $Y$ is aspherical. Consider the two-cell of $W-Z$. Because $n = 4$ or $6$, $m = \pm 1$, and $k \equiv 0$ or $m$ modulo $n$, in the presence of the relation $t^n = 1$, the attaching map for this two-cell is freely homotopic (in $Z \cup W^{(1)}$) to one modeled on the word $yt^{\pm 1}$. Up to homotopy, the inclusion of $Z$ in $W$ is therefore an elementary expansion and in particular a homotopy equivalence. Thus $W$ is aspherical. Given this, the union $Y = X \cup_Z W$ is the union of two aspherical complexes with aspherical intersection. Since the inclusions of $Z$ in $X$ and in $W$ both induce injective homomorphisms on their fundamental groups, $Y$ is aspherical by a theorem of J.~H.~C.~Whitehead~\cite{Whitehead39}. Thus the relative presentation $\pres{H,y}{y^{m-k}gy^kh}$ is aspherical, which completes the proof of Theorem~\ref{mainthm:asphericity}.
\end{proof}


\section{The derived subgroup $J_n(m,k)'$}\label{sec:derived}

We now show that for $n=4$ or $6$ the group $J_n(m,k)'$ has finite abelianization. Together with the fact that $J_n(m,k)^\mathrm{ab}\cong \Z_{nm}$ this proves Theorem~\ref{mainthm:structure}\textit{(a)}. The cases $n = 4$ and $6$ are handled separately. The group $J_n(m,k)'$ is trivial if $m=0$, by Lemma~\ref{lem:abelianJ(m,k)} and since $J_n(m,k)\cong J_n(-m,-k)$ we can assume $m\geq 1$.

\begin{theorem}\label{thm:J'abfinite4}
If $m\geq 1$ then the derived subgroup $J_4(m,k)'$ has a presentation
\begin{alignat}{1}
\pres{x_0,\ldots , x_{4m-1}}{x_ix_{i+m}x_{i+k}^{-1},x_{i}x_{i+m}x_{i+2m}x_{i+3m}\ (0\leq  i < 4m)}\label{eq:J4'bicyclic}
\end{alignat}
and has finite abelianization.
\end{theorem}

\begin{proof} Let $J = J_4(m,k)$. Adding the relation $y^{4m}=1$ (Lemma \ref{lem:conjugacyrelation4}) and using $t^4 = 1$ to rewrite $y^{m-k}t^3y^kt^2 = 1$ as $y^{m-k}t^{-1}y^kt^2 = 1$, the presentation~(\ref{eq:J4'bicyclic}) of $J'$ arises via a straightforward application of the Reidemeister-Schreier process. This presentation shows that $J'$ is a homomorphic image of the cyclically presented group $G = G_{4m}(x_0x_mx_{k}^{-1})$ and so $J/J'$ is a homomorphic image of $G/G'$, which is finite~\cite[Theorem~4]{Williams09}.

A quick proof that $G/G'$ is finite uses the representer polynomial $f_w(x) = 1+x^m-x^{k}$. It suffices to show that if $\lambda$ is a $4m$-th root of unity, then $f_w(\lambda) \neq 0$. Given $\lambda^{4m} = 1$, the quantity $\zeta = \lambda^m$ is a fourth root of unity, so the argument proceeds in cases. If $\zeta = 1$, then $f_w(\lambda) = 1 \neq 0$. If $\zeta = -1$, then $f_w(\lambda) = -\lambda^k \neq 0$. If $\zeta = \pm i$, then $f_w(\lambda) = 1\pm i - \lambda^k$. This last quantity is non-zero because $1\pm i \neq \lambda^k$ due to the fact that these numbers have different complex moduli: $|1\pm i| = \sqrt{2} \neq 1 = |\lambda^k|$. \end{proof}

\begin{theorem}\label{thm:J'abfinite6}
If $m\geq 1$ then the derived subgroup $J_6(m,k)'$ has a presentation
\begin{alignat}{1}
\twolinepres{x_0,\ldots , x_{6m-1}}{x_ix_{i+m}x_{i+2m}(x_{i+k}x_{i+m+k})^{-1},}{ x_{i}x_{i+m}x_{i+2m}x_{i+3m}x_{i+4m}x_{i+5m}\ (0\leq  i < 6m)}\label{eq:J6'bicyclic}
\end{alignat}
and has finite abelianization.
\end{theorem}

\begin{proof} Let $J = J_6(m,k)$. Adding the relation $y^{6m}=1$ (Lemma~\ref{lem:conjugacyrelation6}), a straightforward application of the Reidemeister-Schreier process yields a presentation
\begin{alignat*}{1}
 J' = \twolinepres{x_0,\ldots , x_{6m-1}}{x_ix_{i+m}x_{i+2m}x_{i+3m+k}x_{i+4m+k},}{ x_{i}x_{i+m}x_{i+2m}x_{i+3m}x_{i+4m}x_{i+5m}\ (0\leq  i < 6m)}
 \end{alignat*}
for the derived subgroup.

In the presence of the relations $x_{i}x_{i+m}x_{i+2m}x_{i+3m}x_{i+4m}x_{i+5m} = 1$ the relations $x_ix_{i+m}x_{i+2m}x_{i+3m+k}x_{i+4m+k} = 1$ are equivalent to $x_ix_{i+m}x_{i+2m}=x_{i+k}x_{i+m+k}$ ($0\leq i <6m$) and so we obtain the presentation~(\ref{eq:J6'bicyclic}). This presentation shows that $J'$ is a homomorphic image of the cyclically presented group $G = G_{6m}(x_0x_mx_{2m}(x_{k}x_{m+k})^{-1})$ and so $J/J'$ is a homomorphic image of $G/G'$, which is finite~\cite[Theorem~5.1]{CRS03}. 

A quick proof that $G/G'$ is finite uses the representer polynomial
\[f_w(x) = 1+x^m+x^{2m}-x^{k}-x^{m+k}.\]
It suffices to show that if $\lambda$ is a $6m$-th root of unity, then $f_w(\lambda) \neq 0$. Given $\lambda^{6m} = 1$, the quantity $\zeta = \lambda^m$ is a sixth root of unity, so we argue in cases. If $\zeta = 1$, then $f_w(\lambda) = 3-2\lambda^k \neq 0$. The fact that $3 \neq 2\lambda^k$ is due to the fact that these two numbers have different complex moduli: $|2\lambda^k| = 2 \neq 3 = |3|$.  If $\zeta = -1$, then $f_w(\lambda) = 1 \neq 0$. If $\zeta^3 = 1$ and $\zeta \neq 1$, then $f_w(\lambda) = -\lambda^k(1+\zeta) \neq 0$ because $1+\zeta=-\zeta^2 \neq 0$. And finally, if $\zeta^3 = -1$ and $\zeta \neq -1$, $0 = \zeta^3+1 = (\zeta + 1)(\zeta^2-\zeta+1)$ so $\zeta^2 = \zeta - 1$; thus $f_w(\lambda) = 2\zeta - \lambda^k(1+\zeta)$. This last quantity is non-zero because if $\zeta$ is a primitive sixth root of unity, then $2\zeta$ and $\lambda^k(1+\zeta)$ have different complex moduli: $|\lambda^k(1+\zeta)| = \sqrt{3} \neq 2  = |2\zeta|$.
\end{proof}

\begin{corollary}\label{cor:freeproduct}
If $n = 4$ or $6$ and $d = \gcd(m,k)$, then $J_n(m,k)'$ is the free product of $d$ copies of $J_n(m/d,k/d)'$.\end{corollary}

\begin{proof}
As in, for example,~\cite{Edjvet03}, the presentations~(\ref{eq:J4'bicyclic}), (\ref{eq:J6'bicyclic}) decompose into $d$ disjoint subpresentations of recognizable type.
\end{proof}


We now focus on the case $(m,k)=1$. Theorems~\ref{thm:J4'finite} and~\ref{thm:J6'finite} will show that $J_n(m,k)'$ is finite and abelian. Together with Theorems~\ref{thm:J'abfinite4} and \ref{thm:J'abfinite6}, this shows that $J_n(m,k)$ is metabelian. We also show that if $m-2k\not \equiv 0$~mod~$n$ then $J_n(m,k)'$ is cyclic. To complete the proof of Theorem~\ref{mainthm:structure}\textit{(c)} it will remain to address the case $m-2k\equiv 0$~mod~$n$, which we postpone until Section~\ref{sec:noncyclicJ'}. To begin, recall the following.

\begin{lemma}\label{lemma:cyclicBase} If $\gcd(r,s) = 1$ and $n \geq 1$, then $G_n(x_0^rx_1^{-s}) \cong \Z_\alpha$ where $\alpha = |r^n-s^n|$. Any of the elements $x_i\ (0 \leq i < n)$ can serve as a sole generator.
\end{lemma}

\begin{proof}
 The isomorphism $G = G_n(x_0^rx_1^{-s}) \cong \Z_\alpha$ is noted without proof in~\cite[page 248]{PrideExactlyTwo}. We include a proof for completeness. Using
\[x_i^{r^n} = x_{i+1}^{sr^{n-1}} = \cdots = x_{i+n-1}^{s^{n-1}r} = x_i^{s^n}\]
it follows that $x_i^{\alpha} = 1$ for all $i$. Choosing integers $a, b$ so that $r^na+sb = 1$, we have
\[x_{i+1} = x_{i+1}^{r^na+sb} = x_{i+1}^{s^na+sb} = x_{i+1}^{s(s^{n-1}a+b)} = x_i^{r(s^{n-1}a+b)}.\]
This implies that $G$ is cyclic, generated by any of the $x_i$. Letting $\mu = r(s^{n-1}a+b)$, an isomorphism $G \ra \Z_\alpha$ maps $x_i$ to $\mu^i$ modulo $\alpha$. This is because $\mu^n \equiv 1 \mod \alpha$: 
\begin{eqnarray*}
\mu^n &=& r^n(s^{n-1}a+b)^n\\
&=&r^n \left(s^{n-1}a + \frac{1-ar^n}{s}\right)^n\\
&=&r^n\frac{(s^na+1-ar^n)^n}{s^n}\\
&\equiv &(1-a(r^n-s^n))^n~\mathrm{mod}~\alpha\\
&\equiv& 1~\mathrm{mod}~\alpha,
\end{eqnarray*}
whence the result.
\end{proof}

\begin{corollary}\label{cor:h=ginversefinite}
Suppose $H\cong \Z_n$ and $g,h$ generate $H$ with $g=h^{-1}$. If $(m,k)=1$ then the group~$G$ defined by the relative presentation $\pres{y,H}{y^{m-k}gy^kh}$ is finite of order $n|(m-k)^n-k^n|$.
\end{corollary}

\begin{proof}
If $(m,k)=1$ then the lemma implies that $G_n(x_0^{m-k}x_1^{-k})\cong \Z_{|(m-k)^n-k^n|}$ and thus its $\Z_n$-extension $\pres{y,t}{t^n,y^{m-k}ty^kt^{-1}}$ (which is isomorphic to~$G$) is finite of order $|(m-k)^n-k^n|$.
\end{proof}

\begin{theorem}\label{thm:J4'finite} If $m \neq 0$ and $\gcd(m,k) = 1$, then $J_4(m,k)'$ is finite and abelian.
If $m-2k\not \equiv 0$~mod~$4$ then $J_4(m,k)'$ is cyclic and if $m-2k\equiv 0$~mod~$4$ then $J_4(m,k)'$ is generated by $x_0,x_m$ of the presentation~(\ref{eq:J4'bicyclic}) and the relations $x_i^2=x_{i+2k-m}$ hold $(0\leq i<4m)$.
\end{theorem}

\begin{proof}
As before, assume that $m \geq 1$ and $0 \leq k < 4m$. Setting $J = J_4(m,k)$, Theorem~\ref{thm:J'abfinite4} provides that $J'/J''$ is finite and that $J'$ has the presentation~(\ref{eq:J4'bicyclic}).
We refine this presentation as follows. For all $i$, $1=(x_ix_{i+m})(x_{i+2m}x_{i+3m}) = x_{i+k}x_{2m+k}$ and so re-indexing we can add the relations $x_ix_{i+2m} = 1\  (0\leq  i < 4m)$. In addition, $1 = x_ix_{i+m}x_{i+2m}x_{i+3m} =$ $x_ix_{i+m}x_i^{-1}x_{i+m}^{-1}\  (0\leq  i < 4m)$. Thus $$x_{i+m+k} = x_{i+m}x_{i+2m} = x_{i+m}x_i^{-1}$$ and hence $x_{i+m+k}x_{i+k} = (x_{i+m}x_i^{-1})(x_ix_{i+m}) = x_{i+m}^2$. Re-indexing, this gives $$x_i^2 = x_{i+k}x_{i-m+k} = x_{i-m+k}x_{i+k} = x_{i+2k-m}$$ for $0 \leq i < 4m$. Thus $J'$ has a presentation
\begin{alignat*}{1}
J' = \twolinepres{x_0,\ldots,x_{4m-1}}{x_ix_{i+m} = x_{i+k}, [x_i,x_{i+m}] = 1,}{ x_ix_{i+2m} = 1, x_i^2 = x_{i+2k-m} \ (0 \leq i < 4m)}.
\end{alignat*}
Let $\delta = \gcd(m-2k,4m)$. Using Lemma \ref{lemma:cyclicBase}, the relations $x_i^2=x_{i+2k-m}\ (0 \leq i < 4m)$ imply that if $i \equiv j \mod \delta$, then $x_i$ and $x_j$ each generate the same cyclic subgroup of finite order in $J'$. In particular, each of $x_i$ and $x_j$ is expressible as a power of the other. If $m$ is odd then $\delta=1$ so each $x_i$ can be expressed as a power of $x_0$ so $J_4(m,k)'$ is cyclic. If $m\equiv 0$~mod~$4$ then $\delta =2$ and $k$ is odd, thus each $x_i$ ($i$ even) can be expressed as a power of $x_0$, and each $x_i$ ($i$ odd) can be expressed as a power of $x_k$; since also $x_k=x_0x_m$ we have that $x_k$ is a power of $x_0$, so $J_4(m,k)'$ is cyclic.

Assume then that $m-2k\equiv 0$~mod~$4$ and let $A$ be the subgroup of $J'$ generated by $x_0$ and $x_m$. The relations $x_ix_{i+2m}=1$ and $[x_i,x_{i+m}] = 1$ imply that $A$ is abelian and contains the elements $x_0, x_m, x_{2m}$, and $x_{3m}$. It suffices to show that $J' = A$.

We now have that $4|\delta$ so gcd$(m,\delta)=2$ so $m\Z + \delta \Z = 2\Z$. This means that each even number $i$ is congruent to a multiple of $m$ modulo $\delta$ and so, when $i$ is even then $x_i$ is in the subgroup $A$ generated by $x_0$ and $x_m$. Now $m$ is even and $k$ is odd, so the relations $x_i = x_{i-k}x_{i+m-k}$ show that if $i$ is odd, then $x_i$ is a product of generators with even index, and so $x_i \in A$.
\end{proof}

In the case $m-2k\equiv 0$~mod~$4$, we will show (in Theorem~\ref{thm:(m-2k)=0mod4}) that $J_4(m,k)'$ is the direct product of two cyclic groups of equal order.

\begin{theorem}\label{thm:J6'finite} If $m \neq 0$ and $\gcd(m,k) = 1$, then $J_6(m,k)'$ is finite and abelian, and has the presentation
\begin{alignat}{1}
J' = \twolinepres{x_0,\ldots,x_{6m-1}}{x_i^2 = x_{i+k-m}x_{i+k}, x_ix_{i+2m}=x_{i+m}, [x_i,x_{i+m}],}{ x_i^3=x_{i+m-2k}^4\ (0 \leq i < 6m)}.\label{eq:J6'pres}
\end{alignat}
If $m-2k\not \equiv 0$~mod~$6$ then $J_6(m,k)'$ is cyclic and if $m-2k\equiv 0$~mod~$6$ then $J_6(m,k)'$ is generated by $x_0$ and $x_m$.
\end{theorem}

\begin{proof} As before, assume that $m \geq 1$ and $0 \leq k < 6m$. Setting $J = J_6(m,k)$, by Theorem~\ref{thm:J'abfinite6} we know that $J'/J''$ is finite. Using the expanded presentation for $J$ in Lemma \ref{lem:conjugacyrelation6}, an application of the Reidemeister-Schreier process yields a presentation for $J'$ with generators $x_i = y^ity^{-(i+m)}\ (0 \leq i < 6m)$ and defining relations as follows.
\begin{eqnarray}
x_ix_{i+m}x_{i+2m}x_{i+3m}x_{i+4m}x_{i+5m} &=& 1,\label{eqn:t6}\\
x_ix_{i+m}x_{i+2m}x_{i+3m+k}x_{i+4m+k} &=&1,\label{eqn:W}\\
x_ix_{i+2m}x_{i+m}^{-1}&=&1,\label{eqn:conj}\\
x_ix_{i+m}x_{i+3m}x_{i+4m}&=&1.\label{eqn:square}
\end{eqnarray}
We begin with a series of refinements to this presentation. Using (\ref{eqn:t6}), the relations (\ref{eqn:W}) are equivalent to $x_ix_{i+m}x_{i+2m} = x_{i+k}x_{i+m+k}$. Using (\ref{eqn:conj}), the relations (\ref{eqn:square}) are equivalent to $x_ix_{i+2m}x_{i+4m} = 1$ and then again to $x_ix_{i+3m} = 1$, at which point the relations (\ref{eqn:t6}) become redundant:
\begin{eqnarray*}
(x_ix_{i+m}x_{i+2m})(x_{i+3m}x_{i+4m}x_{i+5m}) &=&x_{i+k}(x_{i+m+k}x_{i+3m+k})x_{i+4m+k}\\
 &=&x_{i+k}x_{i+2m+k}x_{i+4m+k}\\
 &=& x_{i+k}x_{i+3m+k} = 1.
 \end{eqnarray*}
Therefore
\begin{alignat*}{1}
J' = \twolinepres{x_0,\ldots, x_{6m-1}}{x_ix_{i+m}x_{i+2m} = x_{i+k}x_{i+m+k},}{ x_ix_{i+2m} = x_{i+m}, x_ix_{i+3m} = 1\ (0 \leq i < 6m)}.
\end{alignat*}
Next, $x_{i+2m} = x_{i+m}x_{i+3m} = x_{i+m}x_i^{-1}$ so $x_ix_{i+2m} = x_{i+m} = x_{i+2m}x_i$ and so $x_i$ centralizes $x_{i+2m}$. Thus we have that $x_{i+2m}x_{i+3m} = x_{i+2m}x_i^{-1} = x_i^{-1}x_{i+2m}$ $= x_{i+3m}x_{i+2m}$ and hence after re-indexing $[x_i,x_{i+m}] = 1$ for $0 \leq i < 6m$.

Using $[x_i,x_{i+m}] = 1$, the relations $x_ix_{i+m}x_{i+2m} = x_{i+k}x_{i+m+k}$ are equivalent to $x_ix_{i+2m}x_{i+m} = x_{i+k}x_{i+m+k}$ and thence, using $x_ix_{i+2m} = x_{i+m}$, to $x_{i+m}^2 = x_{i+k}x_{i+m+k}$. Further, $x_ix_{i+2m} = x_{i+m}$ implies $x_ix_{i+2m}x_{i+3m} = x_{i+m}x_{i+3m} = x_{i+2m}$ so the relations $[x_i,x_{i+m}] = 1$ imply $x_ix_{i+3m} = 1$. Therefore
\begin{alignat*}{1}
J' = \twolinepres{x_0,\ldots, x_{6m-1}}{x_i^2 = x_{i+k-m}x_{i+k}, x_ix_{i+2m} = x_{i+m},}{ [x_i,x_{i+m}] = 1\ (0 \leq i < 6m)}.
\end{alignat*}
A further family of relations arises from the fact that
\begin{alignat*}{1}
  x_{i+m-2k}^4 = (x_{i+m-2k}^2)^2 &= (x_{i-k}x_{i+m-k})^2 = x_{i-k}^2x_{i+m-k}^2 \\
  &\qquad = (x_{i-m}x_i)(x_ix_{i+m}) = x_i^2x_{i-m}x_{i+m} =   x_i^3.
\end{alignat*}
We have now reached our first goal, which is the presentation~(\ref{eq:J6'pres}) for the derived subgroup.
Let $\delta = \gcd(m-2k,6m)$. Using Lemma~\ref{lemma:cyclicBase}, the relations $x_i^3=x_{i+m-2k}^4$ imply that if $i \equiv j \mod \delta$, then $x_i$ and $x_j$ each generate the same cyclic subgroup of finite order in $J'$. In particular, each of $x_i$ and $x_j$ is expressible as a power of the other.

If $m \not \equiv 2k$~mod~$6$ then $\delta =1,2,3$ or $4$. If $\delta=1$ then each generator can be expressed as a power of $x_0$ so $J'$ is cyclic.
Suppose $\delta=2$. Then $m$ is even and $k,m-k$ are odd. Each $x_i$ ($i$ even) can be expressed as power of $x_0$ and each
$x_i$ ($i$ odd) can be expressed as power of $x_1$. We have the relation $x_1^2=x_{1+k-m}x_{1+k}$ and since $x_{1+k-m},x_{1+k}$ are powers of $x_0$,
so is $x_1^2$. The relation $x_1^3=x_{1+m-2k}^4$ gives that $x_1^3$ is a power of $x_1^2$ and so it is a power of $x_0$. Thus $x_1=x_1^3x_1^{-2}$ is a power of $x_0$ and hence $J'$ is cyclic.

Suppose $\delta =4$. Then $m\equiv 2$~mod~$4$ and $k$ is odd.
Each $x_i$ ($i\equiv 0$~mod~$4$) is a power of $x_0$, each $x_i$ ($i\equiv 1$~mod~$4$) is a power of $x_1$, each $x_i$ ($i\equiv 2$~mod~$4$) is a power of $x_2$ and each $x_i$ ($i\equiv 3$~mod~$4$) is a power of $x_3$. We have the relation $x_{2-m}x_{2+m}=x_2$, and since $x_{2-m},x_{2+m}$ are powers of $x_0$, so is $x_2$. Similarly $x_3$ is a power of $x_1$. We have the relation $x_{1+k-m}x_{1+k}=x_1^2$, and since $x_{1+k-m},x_{1+k}$ are powers of $x_0$, so is $x_1^2$. We have the relation $x_1^3=x_{1+m-k}^4$, and since $x_{1+m-k}$ is a power of $x_1$ we have that $x_1^3$ is a power of $x_1^2$
and so it is a power of $x_0$. Thus $x_1=x_1^3x_1^{-2}$ is a power of $x_0$ and hence $J'$ is cyclic.

Suppose then that $\delta=3$. Then since $(m,k)=1$ we have that $3$ does not divide $m$, so $m\equiv \epsilon$~mod~$3$ ($\epsilon =\pm 1$),
and hence $k\equiv 2\epsilon$~mod~$3$, $m-k\equiv 2\epsilon$~mod~$3$. Each $x_i$ ($i\equiv 0$~mod~$3$) is a power of $x_0$,
each $x_i$ ($i\equiv \epsilon$~mod~$3$) is a power of $x_\epsilon$, each $x_i$ ($i\equiv 2\epsilon$~mod~$3$) is a power of $x_{2\epsilon}$.
The relations $x_0x_m=x_{m-k}^2$ and $x_0x_{2m}=x_m$ imply $x_0^2x_{2m}=x_{m-k}^2$ so $x_0^2=x_{m-k}^2x_{2m}^{-1}$, and since $x_{m-k},x_{2m}$ are powers of $x_{2\epsilon}$, so is $x_0^2$. We have $x_0^3=x_{m-2k}^4$ and since $x_{m-2k}$ is a power of $x_0$ we have that $x_0^3$ is a power of $x_0^2$, and hence it is a power of $x_{2\epsilon}$. Thus $x_0=x_0^3x_0^{-2}$ is a power of $x_{2\epsilon}$. We have the relation $x_{\epsilon-m}x_\epsilon=x_{\epsilon-k}^2$ so $x_\epsilon=x_{\epsilon-m}^{-1}x_{\epsilon-k}^2$, and since $x_{\epsilon-m},x_{\epsilon-k}$ are powers of $x_{2\epsilon}$ we have that $x_\epsilon$ is a power of $x_{2\epsilon}$. Thus $x_0,x_\epsilon$ are powers of $x_{2\epsilon}$ and hence $J'$ is cyclic.

Assume then that $m-2k\equiv 0$~mod~$6$ and let $A$ be the subgroup of $J'$ generated by $x_0$ and $x_m$. The relations $x_ix_{i+2m}=x_{i+m}$ and $[x_i,x_{i+m}] = 1$ imply that $A$ is abelian and contains the elements $x_0, x_m, x_{2m}, x_{3m}, x_{4m}$, and $x_{5m}$. It suffices to show that $J' = A$.

We now have that $\delta=6$ or $12$ and gcd$(m,\delta)=2$. Thus $m\Z + \delta \Z = 2\Z$.
This means that each even number $i$ is congruent to a multiple of $m$ modulo $\delta$ and so as before, if $i$ is even then $x_i$ is in the subgroup $A$ generated by $x_0$ and $x_m$. Now $m$ is even and $k$ is odd, so the relations $x_i^2 = x_{i+k-m}x_{i+k}$ show that if $i$ is odd, then $x_i^2$ is a product of generators with even index, and so $x_i^2 \in A$. At this point we also know that the element $x_{i+m-2k}^2$ lies in $A$. Thus, we have
\[x_i^3 = x_{i+m-2k}^4 = (x_{i+m-2k}^2)^2 \in A\]
and so both $x_i^2$ and $x_i^3$, and hence $x_i$, are in the subgroup $A$, as required.
\end{proof}

Since $J_6(m,k)'$ is generated by $x_0,x_m$, setting $y_i=x_{im}$ ($0\leq i <6$), we have that it is also generated by the $y_i$ and that the relations $y_iy_{i+2}=y_{i+1}$ hold (for $0\leq i<6$). Therefore $J_6(m,k)'$ is a quotient of the \em Sieradski group \em $S(2,6)=G_6(y_0y_2y_1^{-1})$, introduced in~\cite{Sieradski86}, and which is an infinite metabelian group that abelianizes to $\Z \oplus \Z$ (\cite{Thomas91Kim}). In the case $m-2k\equiv 0$~mod~$6$, we will show (in Theorem~\ref{thm:(m-2k)=0mod6}) that $J_6(m,k)'$ is the direct product of two cyclic groups of equal order.

We can now prove Theorem~\ref{mainthm:structure}(d).

\begin{proof}[Proof of Theorem~\ref{mainthm:structure}(d)] Let $d = \gcd(m,k) > 1$. By Corollary~\ref{cor:freeproduct}, $J'$ is the free product of $d$ copies of $J_n(m/d,k/d)'$
which by Theorems~\ref{thm:J'abfinite4},\ref{thm:J'abfinite6},\ref{thm:J4'finite}, and~\ref{thm:J6'finite} is a non-trivial, finite, abelian group $A$.
The derived subgroup $J''$ is therefore a finitely generated free group. By~\cite[Theorem~1]{AnshelPrener72} the rank of $J''$ is equal to $1-d|A|^{d-1}+d|A|^d-|A|^d$. (This can also be proved directly using a standard argument using covering spaces and Euler characteristics.) Since $d>1$ and $|A| > 1$, to show that $J''$ is non-abelian it is enough to show that $|A| \neq 2$. This follows from the fact that $J_n(m/d,k/d)$ is non-abelian: if $A = J_n(m/d,k/d)'$ were cyclic of order two, then by Lemmas~\ref{lem:conjugacyrelation4} and \ref{lem:conjugacyrelation6}, $J_n(m/k,k/d) = A \rtimes \Z_{nm/d}$ would be abelian because the cyclic group of order two has no non-identity automorphism.
\end{proof}


\section{Cyclically presented groups and the order of $(J_n(m,k)')^\mathrm{ab}$}\label{sec:orderabelianization}

The presentations~(\ref{eq:J4'bicyclic}),(\ref{eq:J6'bicyclic}) suggest generalizing the concept of cyclically presented groups, defined at~(\ref{eq:cycpresdef}), to \em bicyclically presented groups $G_r(w,v)$ \em defined by the presentation
\[\pres{x_0,\ldots ,x_{r-1}}{w,\theta(w),\ldots ,\theta^{r-1}(w), v,\theta(v),\ldots ,\theta^{r-1}(v)}.\]
For $n\geq 2$ and even set
\begin{alignat*}{1}
w_n(m,k)&= \left( \prod_{\alpha=0}^{n/2-1} x_{\alpha m} \right)\left( \prod_{\alpha=0}^{n/2-2} x_{k+\alpha m} \right)^{-1}, \\ v_n(m)&=\prod_{\alpha=0}^{n-1}x_{\alpha m},\\
u_n(m)&=x_0x_{nm/2}.
\end{alignat*}
Then for $n=4$ or $6$ the presentations~(\ref{eq:J4'bicyclic}),(\ref{eq:J6'bicyclic}) give
\begin{alignat*}{1}
J_n(m,k)'=G_{nm}(w_n(m,k),v_n(m)).
\end{alignat*}
\begin{lemma}\label{lem:bicyclicutov}
For arbitrary integers $n,m,k$ where $m>0$ and $n\geq 2$ is even we have
\begin{alignat}{1}
G_{nm}(w_n(m,k),v_n(m))^\mathrm{ab}\cong G_{nm}(w_n(m,k),u_n(m))^\mathrm{ab}.\label{eq:newJnmkderivedpres}
\end{alignat}
\end{lemma}

\begin{proof}
Let $w=w_n(m,k), v=v_n(m)$,
\[R_i=\left( \prod_{\alpha=0}^{n/2-1} x_{i+\alpha m} \right)\left( \prod_{\alpha=0}^{n/2-2} x_{i+k+\alpha m}\right)^{-1}, \ S_i=\prod_{\alpha=0}^{n-1}x_{i+\alpha m}\]
(with subscripts taken as integers). Then
\begin{alignat*}{1}
G_{nm}(w,v)^\mathrm{ab}&=\pres{x_0,\ldots , x_{nm-1}}{R_i,S_i\ (0\leq i<nm)}^\mathrm{ab}
\end{alignat*}
(where the subscripts of generators appearing in $R_i,S_i$ are taken mod~$nm$). Now
\begin{alignat*}{1}
\prod_{\alpha=0}^{n/2-2} x_{i+k+\alpha m}
&= \prod_{\alpha=0}^{n/2-1} x_{i+\alpha m}\quad\mathrm{using~relator}~R_i\\
&= \left(\prod_{\alpha=n/2}^{n-1} x_{i+\alpha m}\right)^{-1}\quad\mathrm{using~relator}~S_i\\
&= \left(\prod_{\alpha=n/2}^{n-2} x_{i+k+\alpha m}\right)^{-1}\quad\mathrm{using~relator}~{R_{i+nm/2}}
\end{alignat*}
i.e.
\[\prod_{\alpha=0}^{n/2-2} x_{i+\alpha m}
=\left(\prod_{\alpha=n/2}^{n-2} x_{i+\alpha m}\right)^{-1}.\]
Therefore
\begin{alignat*}{1}
1
&= \prod_{\alpha=0}^{n-1} x_{i+\alpha m}\quad\mathrm{by}~S_i\\
&= \left(\prod_{\alpha=0}^{n/2-2} x_{i+\alpha m}\right) \left( x_{i+(n/2-1)m} \right) \left( \prod_{\alpha=n/2}^{n-2} x_{i+\alpha m} \right) \left( x_{i+(n-1)m} \right) \\
&= \left( \prod_{\alpha=n/2}^{n-2} x_{i+\alpha m} \right)^{-1} \left( x_{i+(n/2-1)m} \right) \left( \prod_{\alpha=n/2}^{n-2} x_{i+\alpha m} \right) \left( x_{i+(n-1)m} \right) \\
&= \left( x_{i+(n/2-1)m} \right) \left( x_{i+(n-1)m} \right)
\end{alignat*}
i.e. $x_{i+(n/2-1)m}x_{i+(n-1)m}=1$ or equivalently $x_{i}x_{i+nm/2}=1$. Thus
\begin{alignat*}{1}
G_{nm}(w,v)^\mathrm{ab}
&=\pres{x_0,\ldots , x_{nm-1}}{R_i,S_i,x_ix_{i+nm/2}\ (0\leq i<nm)}^\mathrm{ab}\\
&=\pres{x_0,\ldots , x_{nm-1}}{R_i,\prod_{\alpha=0}^{n-1}x_{i+\alpha m},x_ix_{i+nm/2}\ (0\leq i<nm)}^\mathrm{ab}\\
&=\twolinepres{x_0,\ldots , x_{nm-1}}{R_i,\prod_{\alpha=0}^{n/2-1}x_{i+\alpha m}x_{i+(\alpha+n/2)m},}{x_ix_{i+nm/2}\ (0\leq i<nm)}^\mathrm{ab}\\
&=\pres{x_0,\ldots , x_{nm-1}}{R_i,x_ix_{i+nm/2}\ (0\leq i<nm)}^\mathrm{ab}\\
&=G_{nm}(w,u_n(m))^\mathrm{ab}
\end{alignat*}
as required.
\end{proof}

We now extend the formula~(\ref{eq:representer}) for the order of the abelianization of a cyclically presented group to that of a bicyclically presented group of the form $G_{nm}(w,u_n(m))$ for an arbitrary word $w$.

\begin{theorem}\label{thm:newbicyclicabelianization}
Let $m,n\geq 1$, where $n$ is even, and let $F_{nm}$ denote the free group on generators $x_0,\ldots ,x_{nm-1}$. If $w$ is any word in $F_{nm}$ and $u=u_n(m)=x_0x_{nm/2}$ then $G_{nm}(w,u)^\mathrm{ab}$ is finite if and only if $\mathrm{Res}(f_w(x),f_u(x))\neq 0$ in which case
\[|G_{nm}(w,u)|^\mathrm{ab}=|\mathrm{Res}(f_w(x),f_u(x))|=\prod_{\lambda^{nm/2}=-1}|f_w(\lambda)|.\]
\end{theorem}

\begin{proof}
For each $0\leq i<nm$ let $a_i$ denote the exponent sum of generator $x_i$ in the word $w$. Then (with subscripts mod~$n$) we have
\begin{alignat*}{1}
G_{nm}(w,u)^{\mathrm{ab}}
&=\pres{x_0,\ldots, x_{nm-1}}{w(x_i,\ldots, x_{i+n-1}), x_ix_{i+nm/2}\ (0\leq i<nm)}^\mathrm{ab}\\
&=\pres{x_0,\ldots, x_{nm-1}}{x_i^{a_0}\ldots x_{i+nm-1}^{a_{nm-1}}, x_ix_{i+nm/2}\ (0\leq i<nm)}^\mathrm{ab}\\
&=\twolinepres{x_0,\ldots, x_{nm-1}}{(x_i^{a_0}\ldots x_{i+nm/2-1}^{a_{nm/2-1}})(x_{i+nm/2}^{a_{nm/2}}\ldots x_{i+nm-1}^{a_{nm-1}}), } {x_ix_{i+nm/2}\ (0\leq i<nm)}^\mathrm{ab}\\
&=\twolinepres{x_0,\ldots, x_{nm-1}}{(x_i^{a_0}\ldots x_{i+nm/2-1}^{a_{nm/2-1}})(x_{i}^{-a_{nm/2}}\ldots x_{i+nm-1}^{-a_{nm-1}}), } {x_ix_{i+nm/2}\ (0\leq i<nm)}^\mathrm{ab}\\
&=\twolinepres{x_0,\ldots, x_{nm-1}}{x_i^{a_0-a_{nm/2}}\ldots x_{i+nm/2-1}^{a_{nm/2-1}-a_{nm-1}},}{x_ix_{i+nm/2}\ (0\leq i<nm)}^\mathrm{ab}\\
&=\pres{x_0,\ldots, x_{nm-1}}{x_i^{b_0}\ldots x_{i+nm/2-1}^{b_{nm/2-1}}, x_ix_{i+nm/2}\ (0\leq i<nm)}^\mathrm{ab}
\end{alignat*}
where $b_i=a_i-a_{i+nm/2}$ for each $0\leq i <nm/2$. 

Let $\tilde{w}=\tilde{w}(x_0,\ldots ,x_{nm-1})=x_0^{b_0}\ldots x_{nm/2-1}^{b_{nm/2-1}}$. Then $f_w(x)=\sum_{i=0}^{nm-1} a_ix^i$, $f_{\tilde{w}}(x)=\sum_{i=0}^{nm/2-1} b_ix^i$ so up to sign, by~(\ref{eq:resultant}), we have
\begin{alignat*}{1}
\mathrm{Res}(f_w,f_u)
&=\mathrm{Res}\left( \sum_{i=0}^{nm-1} a_ix^i, 1+x^{nm/2} \right)\\
&=\prod_{\lambda^{mn/2}=-1}\left( \sum_{i=0}^{nm-1} a_i\lambda^i \right)\\
&=\prod_{\lambda^{mn/2}=-1}\left( \sum_{i=0}^{nm/2-1} a_i\lambda^i + \sum_{j=0}^{nm/2-1} a_{j+nm/2}\lambda^{j+nm/2}\right)\displaybreak[1]\\
&=\prod_{\lambda^{mn/2}=-1}\left( \sum_{i=0}^{nm/2-1} a_i\lambda^i + \lambda^{nm/2}\sum_{j=0}^{nm/2-1} a_{j+nm/2}\lambda^{j}\right)\displaybreak[1]\\
&=\prod_{\lambda^{mn/2}=-1}\left( \sum_{i=0}^{nm/2-1} a_i\lambda^i + (-1)\sum_{i=0}^{nm/2-1} a_{i+nm/2}\lambda^{i}\right)\displaybreak[1]\\
&=\prod_{\lambda^{mn/2}=-1}\left( \sum_{i=0}^{nm/2-1} (a_i-a_{i+nm/2})\lambda^i\right)\displaybreak[1]\\
&=\mathrm{Res}\left( \sum_{i=0}^{nm/2-1} (a_i-a_{i+nm/2})x^i, 1+x^{nm/2}\right)\displaybreak[1]\\
&=\mathrm{Res}\left( \sum_{i=0}^{nm/2-1} b_ix^i, 1+x^{nm/2}\right)\\
&=\mathrm{Res}(f_{\tilde{w}},f_u).
\end{alignat*}
If $b_i=0$ for all $0\leq i<nm/2$ then $G_{nm}(w,u)^\mathrm{ab}\cong \Z^{nm/2}$ and $\mathrm{Res}(f_w,f_u)=\mathrm{Res}(f_{\tilde{w}},f_u)=\mathrm{Res}(0,f_u)=0$ so the theorem holds in this case. Suppose then that $b_j\neq 0$ for some $0\leq j<nm/2$. By subtracting $j$ from the subscripts of the generators $x_i$ (mod~$nm$) we may assume that $b_0\neq 0$. Let $d=\mathrm{max}\{ i\ |\ b_i\neq 0\}$ then $f_{\tilde{w}}(x)=\sum_{i=0}^{nm/2-1} b_ix^i=\sum_{i=0}^{d} b_ix^i$ is a polynomial of degree~$d<nm/2$ and $\tilde{w}=\tilde{w}(x_0,\ldots ,x_{d})=x_0^{b_0}\ldots x_{d}^{b_{d}}$. Then (with subscripts mod~$nm$) we have
\begin{alignat}{1}
G_{nm}(\tilde{w},u)^\mathrm{ab}&=\pres{x_0,\ldots , x_{nm-1}}{x_i^{b_0}\ldots x_{i+d}^{b_{d}},x_ix_{i+nm/2}\ (0\leq i<nm)}^\mathrm{ab}\nonumber\\
&=\threelinepres{x_0,\ldots , x_{nm-1}}{x_i^{b_0}\ldots x_{i+d}^{b_{d}}\ (0\leq i<nm/2),}{x_\beta^{b_0}\ldots x_{\beta+d}^{b_{d}}\ (nm/2\leq \beta<nm),} {x_j=x_{j+nm/2}^{-1}\ (0\leq j<nm/2)}^\mathrm{ab}\nonumber\displaybreak[1]\\
&=\threelinepres{x_0,\ldots , x_{nm-1}}{x_i^{b_0}\ldots x_{i+d}^{b_{d}}\ (0\leq i<nm/2),}{x_{\beta-nm/2}^{-b_0}\ldots x_{\beta+d-nm/2}^{-b_{d}}\,(nm/2\leq \beta<nm),} {x_j=x_{j+nm/2}^{-1}\ (0\leq j<nm/2)}^\mathrm{ab}\nonumber\displaybreak[1]\\
&=\threelinepres{x_0,\ldots , x_{nm-1}}{x_i^{b_0}\ldots x_{i+d}^{b_{d}}\ (0\leq i<nm/2),}{x_{i}^{-b_0}\ldots x_{i+d}^{-b_{d}}\ (0\leq i<nm/2),} {x_j=x_{j+nm/2}^{-1}\ (0\leq j<nm/2)}^\mathrm{ab}\nonumber\\
&=\twolinepres{x_0,\ldots , x_{nm-1}}{x_i^{b_0}\ldots x_{i+d}^{b_{d}}\ (0\leq i<nm/2)} {x_j=x_{j+nm/2}^{-1}\ (0\leq j<nm/2)}^\mathrm{ab}.\label{eq:nmgennmrelN}
\end{alignat}
Since all subscripts are less than $nm$ we can now take them as integers (rather than reduced modulo some base), and we shall do so throughout the rest of the proof. Using the relations $x_d=x_{d+nm/2}^{-1},\ldots , x_{nm/2-1}=x_{nm-1}^{-1}$ we may eliminate generators $x_{d+nm/2},\ldots ,x_{nm-1}$ to get the $(d+nm/2)$-generator, $(d+nm/2)$-relator presentation
\begin{alignat}{1}
G_{nm}(\tilde{w},u)^\mathrm{ab}&=\twolinepres{x_0,\ldots , x_{d+nm/2-1}}{x_i^{b_0}\ldots x_{i+d}^{b_{d}}\ (0\leq i<nm/2),}{x_j=x_{j+nm/2}^{-1}\ (0\leq j<d)}^\mathrm{ab}.\label{eq:sylvPres}
\end{alignat}
Since we have $f_{\tilde{w}}(x)=\sum_{i=0}^d b_ix^i$ is of degree~$d$ and  $f_u(x)=1+x^{nm/2}$ is of degree $nm/2$ we have that $\mathrm{deg}(f_{\tilde{w}})+\mathrm{deg}(f_u)=d+nm/2$. Then we see that the relation matrix of the presentation~(\ref{eq:sylvPres}) is the $(d+nm/2)\times (d+nm/2)$ Sylvester matrix $S=Syl(f_{\tilde{w}},f_u)$ so $\mathrm{det}(S)=\mathrm{Res}(f_{{w}},f_u)$ and the result follows.
\end{proof}

Using Lemma~\ref{lem:bicyclicutov} and Theorem~\ref{thm:newbicyclicabelianization}, in order to calculate $|(J_n(m,k)')^\mathrm{ab}|$ as in Theorem~\ref{mainthm:orders} we must calculate the resultants $\mathrm{Res}(f_w(x),1+x^{nm/2})$ where $w=w_n(m,k)$, which we now do in Theorem~\ref{thm:requiredres}. For use in Corollary~\ref{cor:Fibtypeabelianisations} we also calculate the resultants $\mathrm{Res}(f_w(x),1-x^{nm/2})$. We shall require the following well known result on resultants (see for example~\cite[Lemma~2.1]{Odoni}).

\begin{proposition}\label{prop:resultant}
If $(n,k)=1$ then $\mathrm{Res}(\alpha x^k-\beta,x^n-1)=\beta^n-\alpha^n$.
\end{proposition}

We now give the required resultants.

\begin{theorem}\label{thm:requiredres}
Let
\[f_w(x)=\sum_{i=0}^{n/2-1} x^{im} - \sum_{i=0}^{n/2-2} x^{k+im}\]
and assume $(m,k)=1$.
\begin{itemize}
  \item[(a)] If $n=4$ then
\begin{itemize}
\item[(i)]    $|\mathrm{Res} (f_w(x),x^{nm/2}+1)|= 2^m+1 - 2(\sqrt{2})^m \cos((2k-m)\pi/4)$;
\item[(ii)]   $|\mathrm{Res} (f_w(x),x^{nm/2}-1)|= 2^m-1$.
\end{itemize}

  \item[(b)] If $n=6$ then
\begin{itemize}
  \item[(i)] $|\mathrm{Res} (f_w(x),x^{nm/2}+1)|= 3^m+4^m - 2(2\sqrt{3})^m \cos((m-2k)\pi/6)$;
  \item[(ii)] $|\mathrm{Res} (f_w(x),x^{nm/2}-1)|= 3^m-2^m$.
\end{itemize}
     \end{itemize}
\end{theorem}

\begin{proof}
Throughout this proof, for $N\geq 1$ we shall write $\zeta_N$ to denote $e^{2\pi \sqrt{-1} /N}$. By~(\ref{eq:resultant}) we have
\begin{alignat*}{1}
|\mathrm{Res} (f_w(x), 1+x^{nm/2})| = \prod_{\lambda^{nm/2}=-1} | f_w(\lambda) |= \prod_{\substack{\lambda^{m}=\zeta_n^i,\\i~\mathrm{odd}}} |f_w(\lambda) |= \prod_{\substack{i=0,\\ i~\mathrm{odd}}}^{n-1} |P_i|
\intertext{and}
|\mathrm{Res} (f_w(x), 1-x^{nm/2})| = \prod_{\lambda^{nm/2}=1} |f_w(\lambda)|= \prod_{\substack{\lambda^{m}=\zeta_n^i,\\i~\mathrm{even}}} |f_w(\lambda) |= \prod_{\substack{i=0,\\ i~\mathrm{even}}}^{n-1} |P_i|
\end{alignat*}
where $P_i=\prod_{\lambda^m=\zeta_n^i} f_w(\lambda)$.

Note that if $\lambda^m=1$ then $f_w(\lambda)=( n/2 - (n/2-1)\lambda^k)$ so
\begin{alignat*}{1}
|P_0|&= \left| \prod_{\lambda^m=1}\left( n/2 - (n/2-1)\lambda^k \right) \right|\\
&=| \mathrm{Res} ((n/2-1)x^k-n/2  ,x^m-1)|\\
&= |(n/2-1)^m-(n/2)^m|
 \end{alignat*}
by Proposition~\ref{prop:resultant}. Further, if $\lambda^m=-1$ then $f_w(\lambda) = -\lambda^k$ if $n/2$ is even and $f_w(\lambda)=1$ if $n/2$ is odd. Either way this gives
\[ P_{n/2}=\prod_{\lambda^m=-1} f_w(\lambda)=\pm 1.\]
\noindent (a)(i) We have $|\mathrm{Res} (f_w(x),1+x^{2m})|=|P_1P_3|=|P_1P_{-1}|$ and
   \begin{alignat*}{1}
P_1&=\prod_{\lambda^m=\zeta_4} \left((1+i)-\lambda^k\right),\\
P_{-1}&=\prod_{\lambda^m=\zeta_4^{-1}} \left((1-i)-\lambda^k\right).
   \end{alignat*}
Now $\lambda^m=\zeta_4$ if and only if $\lambda=\zeta_{4m}^{1+4q}$ ($0\leq q<m$) so
\begin{alignat*}{1}
P_1
&=\prod_{q=0}^{m-1} \left( (1+i) -\zeta_{4m}^k(\zeta_{4m}^k)^{4q}\right)\\
&=\prod_{q=0}^{m-1} \left( (1+i) -\zeta_{4m}^k(\zeta_{m}^k)^{q}\right)\\
&=\mathrm{Res}(\zeta_{4m}^k x-(1+i), x^m-1)\\
&=(1+i)^m-(\zeta_{4m}^k)^m\quad \mathrm{by~Proposition~\ref{prop:resultant}}\\
&=(1+i)^m-\zeta_{4}^k\\
&=(\zeta_8\sqrt{2})^m-\zeta_{4}^k.\\
\intertext{Similarly $P_{-1}=(\zeta_8^{-1}\sqrt{2})^m-\zeta_{4}^{-k}$ so}
P_1P_{-1}
&= 2^m-(\sqrt{2})^m(\zeta_8^m\zeta_4^{-k}+\zeta_8^{-m}\zeta_4^k)+1\\
&= 2^m+1-(\sqrt{2})^m(\zeta_8^{m-2k}+\zeta_8^{-(m-2k)})\\
&= 2^m+1 - (\sqrt{2})^m 2\cos((2k-m)\pi/4).
\end{alignat*}

    \noindent (ii) We have $|\mathrm{Res} (f_w(x),1-x^{2m})|=|P_0P_2|$ and we have shown $|P_0|=2^m-1$ and $|P_2|=1$ so we are done.

\noindent (b)(i) We have $|\mathrm{Res} (f_w(x),1+x^{3m})|=|P_1P_3P_5|=|P_1P_5|=|P_1P_{-1}|$ (since $P_3=\pm 1$) and
\begin{alignat*}{1}
P_1&= \prod_{\lambda^m=\zeta_6} \left(1+\zeta_6+\zeta_6^2-\lambda^k(1+\zeta_6)\right),\\
P_{-1}&= \prod_{\lambda^m=\zeta_6^{-1}} \left(1+\zeta_6^{-1}+\zeta_6^{-2}-\lambda^k(1+\zeta_6^{-1})\right).\displaybreak[1]
\end{alignat*}
Now
$\lambda^m=\zeta_6$ if and only if $\lambda=\zeta_{6m}^{1+6q}$ ($0\leq q<m$) so
\begin{alignat*}{1}
P_1&=\prod_{q=0}^{m-1} \left( 1+\zeta_6+\zeta_6^2 -(\zeta_{6m}^k)^{1+6q}(1+\zeta_6)\right)\\
&=\prod_{q=0}^{m-1} \left( 2\zeta_6 -(\zeta_{6m}^k)^{1+6q}(1+\zeta_6)\right),\displaybreak[1]\\
P_{-1}&=\prod_{q=0}^{m-1} \left( 1+\zeta_6^{-1}+\zeta_6^{-2} -(\zeta_{6m}^{-k})^{1+6q}(1+\zeta_6^{-1})\right)\\
&= \prod_{q=0}^{m-1} \left( 2\zeta_6^{-1} -(\zeta_{6m}^{-k})^{1+6q}(1+\zeta_6^{-1})\right),
\end{alignat*}
so
\begin{alignat*}{1}
P_1P_{-1}
&=\prod_{q=0}^{m-1} \left( 7-2\left((\zeta_{6m}^{-k})^{1+6q}(\zeta_6+1)+(\zeta_{6m}^k)^{1+6q}(\zeta_6^{-1}+1)\right)\right)\\
&=\prod_{q=0}^{m-1} \left(7 -2\left( (\zeta_{6m}^{(1+6q)k}\zeta_6^{-1} +\zeta_{6m}^{-(1+6q)k}\zeta_6) +(\zeta_{6m}^{(1+6q)k}+\zeta_{6m}^{-(1+6q)k})        \right)\right)\\
&=\prod_{q=0}^{m-1} \left(7 -2\left( 2\cos\left(\frac{2\pi (1+6q)k}{6m} - \frac{2\pi}{6}\right) +2\cos \left(\frac{2\pi k(6q+1)}{6m}\right)        \right)\right).
\intertext{Now $\cos(\theta-\phi)+\cos(\theta)=2\cos(\phi/2)\cos(\phi/2-\theta)$, so}
P_1P_{-1}
&=\prod_{q=0}^{m-1} \left(7 -4\left(
2\cos\left( \frac{\pi}{6} \right)
\cos \left( \frac{2\pi}{12} -\frac{2\pi k(1+6q)}{6m}\right)
\right)\right)\displaybreak[1]\\
&=\prod_{q=0}^{m-1} \left(7 -4\sqrt{3}
\cos \left( \frac{2\pi (m -2k(1+6q)}{12m}\right)\right)\displaybreak[1]\\
&=\prod_{q=0}^{m-1}
\left(7-2\sqrt{3}(e^{\frac{2\pi \sqrt{-1}}{12m}(m-2k(1+6q))}+e^{\frac{-2\pi \sqrt{-1}}{12m}(m-2k(1+6q))})\right)\displaybreak[1]\\
&=\prod_{q=0}^{m-1} \left(7-2\sqrt{3}(\zeta_{12m}^{m-2k-12qk}+\zeta_{12m}^{-(m-2k-12qk)})\right)\displaybreak[1]\\
&=\prod_{q=0}^{m-1} \left(7-2\sqrt{3}(\zeta_{12m}^{m-2k}\zeta_{m}^{-qk}+\zeta_{12m}^{-(m-2k)}\zeta_{m}^{qk})\right)\\
&=\prod_{q=0}^{m-1} \left(7-2\sqrt{3}(\alpha\zeta_{m}^{-qk}+\alpha^{-1}\zeta_{m}^{qk})\right)
\end{alignat*}
where $\alpha=\zeta_{12m}^{m-2k}$. Thus
\begin{alignat*}{1}
P_1P_{-1}
&=\prod_{q=0}^{m-1}
\left(\left( -\zeta_{m}^{-qk}\alpha^{-1}\right)
\left(2\sqrt{3}\zeta_{m}^{2qk}- 7\alpha\zeta_{m}^{qk}+2\sqrt{3}\alpha^2 \right)\right)\\
&=\prod_{q=0}^{m-1}
\left(\left( -\zeta_{m}^{-qk}\alpha^{-1}\right)
\left(2\zeta_{m}^{qk}-\sqrt{3}\alpha\right)
\left(\sqrt{3}\zeta_{m}^{qk}-2\alpha\right)\right)\\
&=
\prod_{q=0}^{m-1}\left( \zeta_{m}^{-qk}\alpha^{-1}\right)
\prod_{q=0}^{m-1}\left(2\zeta_{m}^{qk}-\sqrt{3}\alpha\right)
\prod_{q=0}^{m-1}\left(\sqrt{3}\zeta_{m}^{qk}-2\alpha\right)\displaybreak[1]\\
&=-\alpha^{-m}\cdot \mathrm{Res}(2x^k-\sqrt{3} \alpha, x^m-1)\cdot \mathrm{Res}(\sqrt{3}x^k-2\alpha,x^m-1)\\
&=-\alpha^{-m} \left( 2^m - (\sqrt{3}\alpha)^m \right) \left( (\sqrt{3})^m - (2\alpha)^m \right)\qquad\mathrm{by~Proposition~\ref{prop:resultant}}\displaybreak[1]\\
&=-\alpha^{-m} \left( (2\sqrt{3})^m - ( (3\alpha)^m+ (4\alpha)^m ) + (2\sqrt 3 \alpha^2)^m \right)\displaybreak[1]\\
&=3^m+4^m-(2\sqrt{3})^m (\alpha^m+\alpha^{-m}) \displaybreak[1]\\
&=3^m+4^m-(2\sqrt{3})^m (\zeta_{12m}^{m(m-2k)}+\zeta_{12m}^{-m(m-2k)})\displaybreak[1]\\
&=3^m+4^m-(2\sqrt{3})^m (\zeta_{12}^{(m-2k)}+\zeta_{12}^{-(m-2k)})\\
&=3^m+4^m-(2\sqrt{3})^m \left(2\cos((m-2k)\pi/6) \right).
\end{alignat*}

    \noindent (ii) We have $|\mathrm{Res} (f_w(x),1-x^{3m})|=|P_0P_2P_4|=|P_0P_2P_{-2}|$ and we have shown $|P_0|=3^m-2^m$ and
\begin{alignat*}{1}
P_2&= \prod_{\lambda^m=\zeta_6^2} -(1+\zeta_6^{2})\lambda^k,\\
P_{-2}&= \prod_{\lambda^m=\zeta_6^{-2}} -(1+\zeta_6^{-2})\lambda^k=\prod_{\lambda^m=\zeta_6^{2}} -(1+\zeta_6^{-2})\lambda^{-k}.
\end{alignat*}
Thus
\[P_2P_{-2}
=\prod_{\lambda^m=\zeta_6^2} (1+\zeta_6^2)(1+\zeta_6^{-2})
=\prod_{\lambda^m=\zeta_6^2} (\zeta_6)(\zeta_6^{-1})
=1.\]
\end{proof}
Theorem~\ref{mainthm:orders} now follows from Lemma~\ref{lem:bicyclicutov}, Theorem~\ref{thm:newbicyclicabelianization}, and Theorem~\ref{thm:requiredres}.
The converse to Lemma~\ref{lem:abelianJ(m,k)} can be obtained as a corollary to Theorem~\ref{mainthm:orders}; that is, if $J_n(m,k)'=1$ then $m=0$ or $m\equiv k$~mod~$nm$ or $k\equiv 0$~mod~$nm$ or $m\equiv 2k$~mod~$nm$; this completes the proof of Theorem~\ref{mainthm:structure}(b).

The Mersenne Prime Conjecture asserts that there are infinitely many primes of the form $2^p-1$. It follows from Theorem~\ref{mainthm:orders} that if $p\geq 5$ is a prime such that $2^p-1$ is prime then $a_4(2p, p~\mathrm{mod}~4)=(2^p-1)^2$. In Theorem~\ref{thm:(m-2k)=0mod4} we show that in fact $$J_4(2p, p~\mathrm{mod}~4)'\cong \Z_{2^p-1}\oplus \Z_{2^p-1}.$$ Thus when $2^p-1$ is a Mersenne prime, the group $J_4(2p, p~\mathrm{mod}~4)'$ is elementary abelian and non-cyclic. In other cases $a_4(m,k)$ may be a Fermat prime.

In~\cite[page~418]{ESTW07} it is remarked that, as with the Mersenne Prime Conjecture, for coprime integers $a,b$, it seems likely that there are infinitely many primes of the form $(a^p-b^p)/(a-b)$. In particular, it is likely that there are infinitely many primes of the form $4^{p}-3^{p}$. Examples are given at~\cite[A059801,A129736]{oeis}. It follows from Theorem~\ref{mainthm:orders} that if $p\geq 7$ is prime such that $4^p-3^p$ is prime then $a_6(2p,p~\mathrm{mod}~6)=(4^p-3^p)^2$. In Theorem~\ref{thm:(m-2k)=0mod6} we show that in fact $J_6(2p,p~\mathrm{mod}~6)'\cong \Z_{4^p-3^p}\oplus \Z_{4^p-3^p}$. The number $4^{p}+3^{p}$ is prime if $p=1,4,16$ (see [38, A081505]); we do not know if these are the only cases. It follows from Theorem~\ref{mainthm:orders} that $a_6(8,1)=(4^4+3^4)^2=337^2$ and $a_6(32,1)=(4^{16}+3^{16})^2=4338014017^2$.
It will follow from Theorem~\ref{thm:(m-2k)=0mod6} that in fact $J_6(8,1)'\cong \Z_{337}\oplus \Z_{337}$ and $J_6(32,1)'\cong \Z_{4338014017}\oplus \Z_{4338014017}$. Thus the family of groups $J_6(m,k)'$ also contains elementary abelian, non-cyclic groups.

Just as numbers of the form $M_m=2^m-1$ are called \em Mersenne numbers \em and prime Mersenne numbers are \em Mersenne primes, \em Gaussian integers of the form $gm_m=(1\pm i)^m-1$ are called \em Gaussian-Mersenne numbers, \em and irreducible Gaussian-Mersenne numbers are called \em Gaussian-Mersenne primes\em . The norm $GM_m$ of $gm_m$ is given by
\[GM_m =\begin{cases}
  2^m-2^{(m+1)/2}+1 & \mathrm{if}\ m\equiv \pm 1~\mathrm{mod}~8\\
  2^m+2^{(m+1)/2}+1 & \mathrm{if}\ m\equiv \pm 3~\mathrm{mod}~8
  \end{cases}\]
and we have that $gm_m$ is a Gaussian-Mersenne prime if and only if its norm $GM_m$ is prime. It follows from Theorem~\ref{mainthm:orders} that $a_4(m,k)=GM_m$ if and only if $(m,k)=1$ and $m+k\equiv 2$~mod~$4$, so each prime $GM_m$ arises as the order of some $J_4(m,k)'$. Further background and a primality test for Gaussian-Mersenne numbers in given in~\cite{BerrizbeitiaIskra10} and a list of known primes $GM_m$ is available at~\cite[A182300]{oeis}.

As a corollary to Theorem~\ref{mainthm:orders} we note the orders of the abelianizations of certain cyclically presented groups that have been considered in the literature. The groups $G_n(x_0x_mx_k^{-1})$ form a class of groups of Fibonacci type and they have been studied by various authors since their introduction in~\cite{JohnsonMawdesley} -- see~\cite{WilliamsRevisited} for a recent survey. The groups $G_n(x_0x_mx_{2m}(x_kx_{m+k})^{-1})$ belong to the class of generalized Fibonacci groups $P(r,n,k,s,q)$, introduced in~\cite{Prishchepov95}.

\begin{corollary}\label{cor:Fibtypeabelianisations}
Let $(m,k)=1$. Then
\begin{itemize}
  \item[(a)] $|G_{4m}(x_0x_mx_k^{-1})^\mathrm{ab}|=(2^m-1)\left(2^m+1-2(\sqrt{2})^m\cos ((2k-m)\pi/4)\right)$;

  \item[(b)] $|G_{6m}(x_0x_mx_{2m}(x_kx_{m+k})^{-1})^\mathrm{ab}|=$ \newline $(3^m-2^m)\left(3^m+4^m-2(2\sqrt{3})^m\cos ((2k-m)\pi/6)\right)$.
\end{itemize}
\end{corollary}

\begin{proof}
Let $w=w_n(m,k), v=v_n(m)$. Then
\begin{alignat*}{1}
|G_{nm}(w)^\mathrm{ab}|
&=|\mathrm{Res} (f_w(x),1-x^{nm})|\\
&=|\mathrm{Res} (f_w(x),1+x^{nm/2})|\cdot |\mathrm{Res} (f_w(x),1-x^{nm/2})|
\end{alignat*}
and the result follows from Theorem~\ref{thm:requiredres}.
\end{proof}


\section{Isomorphisms}\label{sec:isoms}

We first prove Theorem~\ref{thm:abc} (from the introduction) which asserts that $J_4(m,k)$ is isomorphic to $F^{m,k,m-k}$ and gives a similar presentation for the groups $J_6(m,k)$.

\begin{proof}[Proof of Theorem~\ref{thm:abc}]
\begin{alignat*}{1}
J_4(m,k)&=\pres{t,y}{t^4,y^{m-k}t^3y^kt^2}\\
&= \pres{t,y,R,S}{t^4,y^{m-k}t^3y^kt^2,R=t^2,S=y^{-1}}\\
&= \pres{t,R,S}{R^2,S^{k-m}tRS^{-k}R,R=t^2} \displaybreak[1]\\
&= \pres{t,R,S}{R^2,t=RS^{m-k}RS^k, R=t^2} \\
&= \pres{R,S}{R^2, R=RS^{m-k}RS^kRS^{m-k}RS^k} \\
&= \pres{R,S}{R^2, RS^{m}RS^kRS^{m-k}}=F^{m,k,m-k}.\\
J_6(m,k)&=\pres{t,y}{t^6,y^{m-k}t^3y^kt^2}\\
&= \pres{t,y,R,S}{t^6,y^{m-k}t^3y^kt^2,R=t^3,S=y} \\
&= \pres{t,R,S}{R^2,S^{m-k}RS^{k}Rt^{-1},R=t^3} \\
&= \pres{t,R,S}{R^2,t=S^{m-k}RS^kR, R=t^3} \\
&= \pres{R,S}{R^2, R=S^{m-k}RS^kRS^{m-k}RS^kRS^{m-k}RS^kR} \\
&= \pres{R,S}{R^2, RS^mRS^kRS^{m-k}RS^kRS^{m-k}}.
\end{alignat*}
\end{proof}

We now turn our attention to proving Theorem~\ref{mainthm:Isos}. We start with

\begin{lemma}\label{lem:easyisom}
Let $n=4$ or $6$. Then $J_n(m,k)\cong J_n(m,m-k)$.
\end{lemma}

\begin{proof}
Replacing generators $y,t$ by their inverses then inverting and cyclically permuting the second relator transforms the presentation of the group $J_n(m,k)$ into a presentation of $J_n(m,m-k)$.
\end{proof}

\begin{lemma}\label{lem:isomepsilon}
Let $n=4$ or $6$ and suppose $(m,k)=1$. If $k\equiv \epsilon$~mod~$n$ where $\epsilon= \pm 1$ then $J_n(m,k)\cong J_n(m,\epsilon)$.
\end{lemma}

\begin{proof}
Recall from Lemmas~\ref{lem:conjugacyrelation4} and~\ref{lem:conjugacyrelation6} that $y^{nm}=1$ holds in $J_n(m,k)$. We have that $(k,n)=1$ and $(m,k)=1$ so $(k,nm)=1$. Therefore there exists $\bar{k}$ such that $k\bar{k}\equiv 1$~mod~$nm$. Then
\begin{alignat*}{1}
J_n(m,k)
&=\pres{y,t}{t^n,y^{nm},y^{m-k}t^3y^kt^2}\\
&=\pres{y,t,u}{t^n,y^{nm},y^{m-k}t^3y^kt^2,u=y^{\epsilon k}}.\\
\intertext{Now $y=y^1=y^{k\bar{k}}=u^{\epsilon \bar{k}}$ so}
J_n(m,k)
&=\pres{y,t,u}{t^n,y^{nm},y^{m-k}t^3y^kt^2,u=y^{\epsilon }k,y=u^{\epsilon \bar{k}}}\\
&=\pres{t,u}{t^n,u^{\epsilon nm\bar{k}}, u^{\epsilon m\bar{k}-\epsilon k\bar{k}}t^3u^{\epsilon k\bar{k}}t^2,u=u^{\bar{k}k}}\\
&=\pres{t,u}{t^n,u^{nm\bar{k}}, u^{\epsilon m\bar{k}-\epsilon k\bar{k}}t^3u^{\epsilon k\bar{k}}t^2,u^{\bar{k}k-1}}.\\
\intertext{Now $k\bar{k}\equiv 1$~mod~$nm$ so $k\bar{k}-1=nm\alpha$ for some $\alpha$ where $(\alpha,\bar{k})=1$. So}
J_n(m,k)
&=\pres{t,u}{t^n,u^{nm\bar{k}}, u^{\epsilon m\bar{k}-\epsilon k\bar{k}}t^3u^{\epsilon k\bar{k}}t^2,u^{nm\alpha}}\\
&=\pres{t,u}{t^n,u^{nm(\bar{k},\alpha)}, u^{\epsilon m\bar{k}-\epsilon k\bar{k}}t^3u^{\epsilon k\bar{k}}t^2}\\
&=\pres{t,u}{t^n,u^{nm}, u^{\epsilon m\bar{k}-\epsilon }t^3u^\epsilon t^2}.\\
\intertext{Now $k\equiv \epsilon$~mod~$n$ implies that $k\bar{k}\equiv \epsilon \bar{k}$~mod~$n$ so $\epsilon\bar{k}\equiv 1$~mod~$n$ so $m\epsilon \bar{k}\equiv m$~mod~$n$. Therefore}
J_n(m,k)
&=\pres{t,u}{t^n,u^{nm}, u^{m-\epsilon }t^3u^\epsilon t^2}\\
&=J_n(m,\epsilon).
\end{alignat*}
\end{proof}

\begin{lemma}\label{lem:isomk=3}
If $(m,6)=1$, $(m,k)=1$ and $k\equiv 3$~mod~$6$ then $J_6(m,k)\cong J_6(m,3)$.
\end{lemma}

\begin{proof}
Note that $k/3$ is odd. If $k/3\equiv 1$~mod~$6$ then let $\bar{k}=k+6m$. This gives $J_6(m,k)\cong J_6(m,\bar{k})$ and $\bar{k}/3 \not \equiv 1$~mod~$6$. Thus we may assume $k/3\not \equiv 1$~mod~$6$ so $k/3\equiv 3$ or $5$~mod~$6$. Let $\epsilon=\pm 1$ be defined by
\[\epsilon =\begin{cases}
  -m~\mathrm{mod}~6 & \mathrm{if}\ k/3\equiv 3~\mathrm{mod}~6,\\
  m~\mathrm{mod}~6 & \mathrm{if}\ k/3\equiv 5~\mathrm{mod}~6,
\end{cases}\]
and set $p=2m\epsilon +k/3$. Then it follows that $p\equiv 1$~mod~$6$, $mp\equiv m$~mod~$6m$ and
\begin{alignat}{1}
3p&\equiv k~\mathrm{mod}~6m\label{eq:3p=k}
\intertext{whence}
(m-3)p&\equiv m-k~\mathrm{mod}~6m.\label{eq:(m-3)p}
\end{alignat}
Further, $p\equiv k/3$~mod~$m$ so $(p,m)=(k/3,m)=1$ and hence $(p,6m)=1$.

Now
\begin{alignat*}{1}
J_6(m,k)
&=\pres{y,t}{t^6,y^{6m},y^{m-k}t^3y^kt^2}\\
&=\pres{y,t,u}{t^6,y^{6m},y^{m-k}t^3y^kt^2,u=y^p}.\\
\intertext{Then $y^k=u^3$ by~(\ref{eq:3p=k}) and $y^{m-k}=u^{m-3}$ by~(\ref{eq:(m-3)p}) so}
J_6(m,k)
&=\pres{y,t,u}{t^6,y^{6m},u^{m-3}t^3u^3t^2,u=y^p}.\\
\intertext{Now $(p,6m)=1$ implies that there exists $\alpha,\beta\in\Z$ with $(\alpha,\beta)=1$ such that $6m\alpha +p\beta=1$ so $\beta p\equiv 1$~mod~$6m$, so $y=u^\beta$ and hence}
J_6(m,k)
&=\pres{y,t,u}{t^6,y^{6m},u^{m-3}t^3u^3t^2,u=y^p,y=u^\beta}\\
&=\pres{t,u}{t^6,u^{6m\beta},u^{m-3}t^3u^3t^2,u=u^{p\beta}}\\
&=\pres{t,u}{t^6,u^{6m\beta},u^{m-3}t^3u^3t^2,u^{p\beta-1}}\\
&=\pres{t,u}{t^6,u^{6m\beta},u^{m-3}t^3u^3t^2,u^{6m\alpha}}\\
&=\pres{t,u}{t^6,u^{6m(\alpha,\beta)},u^{m-3}t^3u^3t^2}\\
&=\pres{t,u}{t^6,u^{6m},u^{m-3}t^3u^3t^2}\\
&=J_6(m,3).
\end{alignat*}
\end{proof}

\begin{proof}[Proof of Theorem~\ref{mainthm:Isos}.]
If the groups $J_n(m_1,k_1)$ and $J_n(m_2,k_2)$ are isomorphic, then so are their abelianizations,
ie $\Z_{nm_1}\cong\Z_{nm_2}$, so $m_1=m_2=m$ (say). Further we have that $|(J_n(m,k_1)')^\mathrm{ab}|=|(J_n(m,k_2)')^\mathrm{ab}|$ so by Theorem~\ref{mainthm:orders} we have $$\cos ((2k_1-m)\pi/n)=\cos ((2k_2-m)\pi/n)$$ so $(2k_1-m)\equiv \pm (2k_2-m)$~mod~$2n$. That is, $k_1\equiv k_2$~mod~$n$ or $k_1+k_2\equiv m$~mod~$n$.

Suppose then that $m_1=m_2=m$ (say) and that either $k_1\equiv k_2$~mod~$n$ or $k_1+k_2\equiv m$ mod~$n$. Replacing $k_1$ by $m-k_1$ (which we may do since by Lemma~\ref{lem:easyisom} we have $J_n(m,k_1)\cong$ $J_n(m,m-k_1)$) the second condition becomes the first condition so without loss of generality we may assume $k_1\equiv k_2$~mod~$n$.

Let $k_1'=m-k_1$, $k_2'=m-k_2$. Then $k_1'\equiv k_2'$~mod~$n$ and by Lemma~\ref{lem:easyisom} we have $J_n(m,k_1)\cong J_n(m,k_1')$ and $J_n(m,k_2)\cong J_n(m,k_2')$. If $k_1\equiv \epsilon$~mod~$n$ ($\epsilon =\pm 1$) then Lemma~\ref{lem:isomepsilon} implies that $J_n(m,k_1)\cong J_n(m,\epsilon)$ and $J_n(m,k_2)\cong J_n(m,\epsilon)$ so $J_n(m,k_1)\cong$ $J_n(m,k_2)$. Similarly, if $k_1'\equiv \epsilon$~mod~$n$ then  $J_n(m,k_1')\cong J_n(m,k_2')$ so $J_n(m,k_1)\cong J_n(m,k_2)$. Thus we may assume that $k_1,k_1'\not \equiv \pm 1$~mod~$n$. The fact that $(m,k_1)=1$ implies that either $k_1$ or $k_1'$ is odd. Since $k_1,k_1'\not \equiv \pm 1$~mod~$n$ it follows that $n=6$ and either $k_1\equiv 3$~mod~$6$ or $k_1'\equiv 3$~mod~$6$ which in turn implies that $(m,6)=1$. By Lemma~\ref{lem:isomk=3} we have that $J_6(m,k_1)\cong J_6(m,3)$. It also gives that either $k_2\equiv 3$~mod~$6$ or $k_2'\equiv 3$~mod~$6$ and so $J_6(m,k_2)\cong J_6(m,3)$, so $J_6(m,k_1)\cong J_6(m,k_2)$, as required.
\end{proof}

Theorem~\ref{mainthm:Isos} implies that if $(m,k)=1$ then $J_4(m,k)$ is isomorphic to either $J_4(m,1)$ or $J_4(m,-1)$. Further, $J_4(m,1)\cong J_4(m,-1)$ if and only if $m\equiv 0$~mod~$4$. Likewise, if $(m,k)=1$ then $J_6(m,k)$ is isomorphic to one of $J_6(m,1),J_6(m,-1)$ or $J_6(m,3)$. If $m$ is odd these groups are pairwise non-isomorphic. If $m\equiv 0$~mod~$6$ we have $J_6(m,1)\cong J_6(m,-1)$ and $(m,3)=3$ so only one isomorphism type of finite group occurs. When $m\equiv 2$~mod~$6$ we have $J_6(m,1)\not \cong J_6(m,-1)\cong J_6(m,3)$ and when $m\equiv 4$~mod~$6$ we have $J_6(m,-1)\not \cong$ $J_6(m,1)\cong J_6(m,3)$ so only two isomorphism types of finite groups occur.


\section{Non-metacyclic groups $J_n(m,k)$}\label{sec:noncyclicJ'}

In Theorems~\ref{thm:J4'finite} and~\ref{thm:J6'finite} it was shown that, for $n=4$ or $6$, if $(m,k)=1$ and $m-2k\not \equiv 0$ mod~$n$ then $J_n(m,k)'$ is cyclic. We now show that if $(m,k)=1$ and $m-2k\equiv 0$~mod~$n$ then $J_n(m,k)'$ is the direct product of two cyclic groups of equal order and hence $J_n(m,k)$ is not metacyclic, thus completing the proof of Theorem~\ref{mainthm:structure}(c).

\begin{theorem}\label{thm:(m-2k)=0mod4}
Suppose $(m,k)=1$ and $m-2k\equiv 0$~mod~$4$. Then
\[J_4(m,k)'\cong \begin{cases}
\Z_{2^{m/2}-1}\oplus \Z_{2^{m/2}-1}& \mathrm{if}~m-2k\equiv 0~\mathrm{mod}~8,\\
\Z_{2^{m/2}+1}\oplus \Z_{2^{m/2}+1}& \mathrm{if}~m-2k\equiv 4~\mathrm{mod}~8.
\end{cases}\]
\end{theorem}

\begin{proof}
By Theorem~\ref{mainthm:orders} the order $|J_4(m,k)'|=(2^{m/2}+\epsilon)^2$ where $\epsilon =-1$ if $m-2k\equiv 0$ mod~$8$ and $\epsilon=1$ if $m-2k\equiv 4$~mod~$8$. By Theorem~\ref{thm:J4'finite}  we have that $J_4(m,k)'$ is abelian and is generated by $x_0$ and $x_m$. We show that $x_i^{2^{m/2}+\epsilon}=1$ for any $0\leq i< 4m$, so in particular $x_0^{2^{m/2}+\epsilon}=1$ and $x_m^{2^{m/2}+\epsilon}=1$, as this implies that $J_4(m,k)'\cong \Z_{2^{m/2}+\epsilon}\oplus\Z_{2^{m/2}+\epsilon}$. Note that the order of $J_4(m,k)'$ implies that $x_i^{(2^{m/2}+\epsilon)^2}=1$. By Theorem~\ref{thm:J4'finite}  we have that the relation $x_i=x_{i+2k-m}^2$ holds in $J_4(m,k)'$.

Suppose that $m-2k\equiv 0$~mod~$8$ and let $k'=m/2-4$. Now $m-2k\equiv 0$~mod~$8$ implies that $k\equiv m/2$~mod~$4$ so $k+k'\equiv m$~mod~$4$ and hence $J_4(m,k)\cong J_4(m,k')$ by Theorem~\ref{mainthm:Isos}. The relation $x_{i}=x_{i+2k'-m}^2$ is simply $x_i=x_{i-8}^2$ so
\[x_i=x_{i-8}^{2}=x_{i-16}^{2^2}=\ldots =x_{i-8m/2}^{2^{m/2}}=x_{i}^{2^{m/2}}\]
and thus $x_i^{2^{m/2}-1}=1$.

Suppose then that $m-2k\equiv 4$~mod~$8$ and let $k'=m/2-2$. Now $m-2k\equiv 4$~mod~$8$ implies that $k\equiv 2+m/2$~mod~$4$ so $k+k'\equiv m$~mod~$4$ and hence $J_4(m,k)\cong J_4(m,k')$ by Theorem~\ref{mainthm:Isos}. The relation $x_{i}=x_{i+2k'-m}^2$ is simply $x_i=x_{i-4}^2$ so
\[x_i=x_{i-4}^2=x_{i-8}^{2^2}=\ldots =x_{i-4m}^{2^m}=x_{i}^{2^{m}}\]
and thus $x_i^{2^{m}-1}=1$. But we also have $x_i^{(2^{m/2}+1)^2}=1$ and
\[\left( (2^{m/2}+1)^2, 2^{m}-1\right)= (2^{m/2}+1) \left( 2^{m/2}+1, 2^{m/2}-1\right)=2^{m/2}+1\]
so $x_i^{2^{m/2}+1}=1$.
\end{proof}

\begin{theorem}\label{thm:(m-2k)=0mod6}
Suppose $(m,k)=1$ and $m-2k\equiv 0$~mod~$6$.Then
\[J_6(m,k)'\cong \begin{cases}
\Z_{4^{m/2}-3^{m/2}}\oplus \Z_{4^{m/2}-3^{m/2}}& \mathrm{if}~m-2k\equiv 0~\mathrm{mod}~12,\\
\Z_{4^{m/2}+3^{m/2}}\oplus \Z_{4^{m/2}+3^{m/2}}& \mathrm{if}~m-2k\equiv 6~\mathrm{mod}~12.
\end{cases}\]
\end{theorem}

\begin{proof}
By Theorem~\ref{mainthm:orders} the order $|J_6(m,k)'|=(4^{m/2}+\epsilon 3^{m/2})^2$ where $\epsilon =-1$ if $m-2k\equiv 0$~mod~$12$ and $\epsilon=1$ if $m-2k\equiv 6$~mod~$12$. By Theorem~\ref{thm:J6'finite}  we have that $J_6(m,k)'$ is abelian and is generated by $x_0$ and $x_m$. We show that $x_i^{4^{m/2}+\epsilon 3^{m/2}}=1$ for any $0\leq i< 6m$, so in particular $x_0^{4^{m/2}+\epsilon 3^{m/2}}=1$ and $x_m^{4^{m/2}+\epsilon 3^{m/2}}=1$, as this implies
that $J_6(m,k)'\cong \Z_{4^{m/2}+\epsilon 3^{m/2}}\oplus\Z_{4^{m/2}+\epsilon 3^{m/2}}$. Note that the order of $J_6(m,k)'$ implies that $x_i^{(4^{m/2}+\epsilon 3^{m/2})^2}=1$. By Theorem~\ref{thm:J6'finite}  we have that the relation $x_i^3=x_{i+m-2k}^4$ holds in $J_6(m,k)'$.

Suppose that $m-2k\equiv 0$~mod~$12$ and let $k'=m/2-6$. Now $m-2k\equiv 0$~mod~$12$ implies that $k\equiv m/2$~mod~$6$ so $k+k'\equiv m$~mod~$6$ and hence $J_6(m,k)\cong J_6(m,k')$ by Theorem~\ref{mainthm:Isos}. The relation $x_{i}^3=x_{i+m-2k'}^4$ is simply $x_i^3=x_{i+12}^4$ so
\[x_i^{3^{m/2}}=x_{i+12}^{3^{m/2-1}\cdot 4}=x_{i+24}^{3^{m/2-2}\cdot 4^2}=\ldots =x_{i+12m/2}^{4^{m/2}}=x_{i}^{4^{m/2}}\]
and thus $x_i^{4^{m/2}-3^{m/2}}=1$.

Suppose then that $m-2k\equiv 6$~mod~$12$ and let $k'=m/2-3$. Now $m-2k\equiv 6$~mod~$12$ implies that $k\equiv 3+m/2$~mod~$6$ so $k+k'\equiv m$~mod~$6$ and hence $J_6(m,k)\cong J_6(m,k')$ by Theorem~\ref{mainthm:Isos}. The relation $x_{i}^3=x_{i+m-2k'}^4$ is simply $x_i^3=x_{i+6}^4$ so
\[x_i^{3^{m}}=x_{i+6}^{3^{m-1}\cdot 4}=x_{i+12}^{3^{m-2}\cdot 4^2}=\ldots =x_{i+6m}^{4^{m}}=x_{i}^{4^{m}}\]
and thus $x_i^{4^{m}-3^{m}}=1$. But we also have $x_i^{(4^{m/2}+3^{m/2})^2}=1$ and
\begin{alignat*}{1}
\left( (4^{m/2}+3^{m/2})^2, 4^{m}-3^{m}\right)&= (4^{m/2}+3^{m/2}) \left( 4^{m/2}+3^{m/2}, 4^{m/2}-3^{m/2}\right)\\
&=4^{m/2}+3^{m/2}
\end{alignat*}
so $x_i^{4^{m/2}+3^{m/2}}=1$.
\end{proof}


\section{Applications to generalized Fibonacci groups}\label{sec:Fibonacci}

The \em generalized Fibonacci groups \em
\[F(r,n,l,s)=G_n((x_0x_1\ldots x_{r-1})(x_{(l-1)+r}x_{(l-1)+r+1}\ldots x_{(l-1)+r+s-1})^{-1})\]
were introduced in~\cite{CRAustral} as a generalization of the groups $H(r,n,s)=F(r,n,1,s)$ of~\cite{CampbellRobertsonLMS75}, and of the groups $F(r,n,l)=F(r,n,l,1)$ of~\cite{CampbellRobertsonEMS75}, which in turn generalize the Fibonacci groups $F(r,n)$ of~\cite{JWW}. In Corollary~\ref{maincor:Fibonacci} we use Theorems~\ref{mainthm:structure},\ref{mainthm:orders},\ref{mainthm:Isos} to obtain infinite families of finite metacyclic generalized Fibonacci groups.

\begin{maincorollary}\label{maincor:Fibonacci}
Let $n=4$ or $6$, $m\geq 1$ and suppose $m\equiv \pm 1$~mod~$n$. Then $J_n(m,k)$ has a unique normal subgroup $N$ of index~$n$. Moreover, $N$ is metacyclic and
\begin{itemize}
  \item[(a)] if $n=4$ or $6$ and $J_n(m,k)\cong J_n(m,-1)$ then $N\cong F(m+1,n,n/2)$;
  \item[(b)] if $n=4$ and $J_n(m,k)\cong J_n(m,1)$ then $$N\cong \begin{cases}
    F(m+3,4,3,3) &\mathrm{if}~m\equiv 1~\mathrm{mod}~4,\\
    H(m+2,4,2) &\mathrm{if}~m\equiv -1~\mathrm{mod}~4;
  \end{cases}$$
  \item[(c)] if $n=6$ and $J_n(m,k)\cong J_n(m,1)$ then $$N\cong \begin{cases}
    F(m+5,6,4,5) &\mathrm{if}~m\equiv 1~\mathrm{mod}~6,\\
    F(m+2,6,2,2) &\mathrm{if}~m\equiv -1~\mathrm{mod}~6;
    \end{cases}$$
  \item[(d)] if $n=6$ and $J_n(m,k)\cong J_n(m,3)$ then $N\cong H(m+3,6,3)$.
\end{itemize}
The order of $N$ can be found from Theorem~\ref{mainthm:orders}.
\end{maincorollary}

\begin{proof}
Let $\epsilon = \pm 1$ so $m \equiv \epsilon \mod n$.  Let $N$ be a normal subgroup of index $n$ in $J = J_n(m,k)$. The fact that $m \equiv \epsilon \mod n$ implies that $y^mN = y^\epsilon N$. Next, we claim that $Q = J/N$ is cyclic. Consider the element $tN \in Q$. If $t \in N$, then $Q = \gpres{yN}$ is cyclic. If $t^2 \in N$, then $y^mt^3 \in N$ so $y^{-\epsilon}N = t^3N$ and so $Q = \gpres{t^3N}$ is cyclic. If $t^3 \in N$, then $y^mt^2 \in N$ so $y^{-\epsilon}N = t^2N$ and so $Q = \gpres{t^2N}$ is cyclic. If none of $t, t^2$, or $t^3$ is in $N$, then the element $tN$ has order $n$ in $Q = J/N$, which is therefore cyclic in all cases. Now the fact that $Q = J/N$ is abelian implies that $y^mt^5 \in N$ so $y^{-\epsilon}N = t^{\pm 1}N$ depending on whether $n = 4$ or $6$. Thus $Q = J/N$ is cyclic of order $n$, generated by $tN$, and the element $yN$ is completely determined by $\epsilon \equiv m \mod n$.  Since the quotient homomorphism $J \ra J/N$ is completely determined, so is its kernel $N$. By Theorem~\ref{mainthm:structure} the group $J$ is metacyclic, and hence so is $N$.

Setting $E=E(r,n,l,s)=F(r,n,l,s)\rtimes_\theta \Z_n$, where $\theta(x_i)=x_{i+1}$ we have that $F(r,n,k,l)$ is a normal subgroup of index~$n$ in $E(r,n,l,s)$ and
\[ E(r,n,l,s)=\pres{y,t}{t^n,y^rt^{s+(l-1)}y^{-s}t^{-r-(l-1)}}.\]
For certain choices of the parameters $E(r,n,l,s)$ is isomorphic to a group $J=J_n(m,k)$, and so $F(r,n,l,s)$ is the unique normal subgroup of $J$.
In part~(a) we have that $E(m+1,n,n/2,1)\cong$ $J_n(m,-1)$;
in part~(b), if $m\equiv 1$~mod~$4$ then $E(m+3,4,3,3)\cong J_4(m,-3) \cong J_4(m,1)$, by Theorem~\ref{mainthm:Isos}
and if $m\equiv -1$~mod~$4$ then $E(m+2,4,1,2)\cong J_4(m,-2) \cong J_4(m,1)$;
in part~(c), if $m\equiv 1$~mod~$6$ then $E(m+5,6,4,5)\cong J_6(m,-5) \cong J_6(m,1)$,
and if $m\equiv -1$~mod~$4$ then $E(m+2,6,2,2)\cong J_6(m,-2) \cong J_6(m,1)$;
in part~(d) we have that $E(m+3,6,1,3)\cong J_6(m,-3) \cong J_6(m,3)$.
\end{proof}

For odd $r$, the groups $H(r,4,2)$ were shown to be finite metacyclic groups in~\cite{Brunner} where the orders are also given, thus when $m\equiv -1$~mod~$4$ the groups in part~(b) were already known to be finite. Similarly, in the case $n=4$, $m\equiv 1$~mod~$4$ the groups of part~(a) are the groups $F(4\alpha+2,4,2)$ ($\alpha\geq 4$) and these are isomorphic to the groups $F(4\alpha+2,4)$ which were shown to be finite metabelian groups in~\cite{Seal}. Moreover, it was shown in~\cite{ThomasCommAlg} that $F(4\alpha+2,4)\cong H(4\alpha+3,4,2)$ and so $F(4\alpha+2,4)$ is a finite metacyclic group. The group $F(4,4,2)$ (which occurs in part~(a)) was observed in~\cite{CRToddCoxeter} to have order 39. The group $F(6,6,3)$ (which occurs in part~(a)) was identified in~\cite{CRRT} as being finite of order 10655. All other finite groups in Corollary~\ref{maincor:Fibonacci}
appear to be new observations.

\begin{acknowledgements}
We would like to thank Clayton Petsche for helpful conversations, Martin Edjvet for insightful comments on a draft of this article, Alex Fink for pointing out a mistake in an early version of Theorem~\ref{thm:newbicyclicabelianization}, and the referee for the careful reading of this article.
\end{acknowledgements}


\end{document}